\documentclass{amsart}
\usepackage{amsmath,amssymb,amsthm,bm,bbm}
\usepackage{graphicx,enumerate}
\usepackage{layout}
\usepackage[all,cmtip]{xy}
\usepackage{multicol}

\usepackage{hyperref}
\usepackage{url}

\theoremstyle{plain}
\newtheorem{theorem}{Theorem}[section]

\newtheorem{corollary}[theorem]{Corollary}
\newtheorem{conjecture}[theorem]{Conjecture}

\theoremstyle{definition}

\newtheorem{remark}[theorem]{Remark}

\newtheorem*{acknowledgment}{Acknowledgment}
\theoremstyle{remark}

\newcommand{\bZ}{\mathbbm{Z}}
\newcommand{\bQ}{\mathbbm{Q}}

\newcommand{\Cl}{\mathrm{Cl}}

\title{Simplest cubic fields with small class number}

\author[A. Hoshi]{Akinari Hoshi}
\address{Department of Mathematics, Niigata University, Niigata 950-2181, Japan}
\email{hoshi@math.sc.niigata-u.ac.jp}

\author[H. Iida]{Hiroaki Iida}
\address{Graduate School of Science and Technology, Niigata University, Niigata 950-2181, Japan}
\email{iida.hiroaki0616@gmail.com}

\thanks{This work was partially supported by JSPS KAKENHI Grant Numbers 
19K03418, 24K00519.}
\begin{document}

\begin{abstract}
Let $m\in\bZ$ be an integer and $L_m=\bQ(\alpha)$ be 
the simplest cubic field 
with class number $h_m$ and conductor $\mathfrak{f}_m$ 
where $\alpha$ is a root of $f_m(X)=X^3-mX^2-(m+3)X-1$. 
Let $\mathcal{O}_{L_m}$ be the ring of integers of $L_m$. 
By using PARI/GP, we determine that if $[\mathcal{O}_{L_m}:\bZ[\alpha]]=1$ 
$($resp. $3$, $27$$)$, 
i.e. $m^2+3m+9=\mathfrak{f}_m$ $($resp. $3\mathfrak{f}_m$, 
$27\mathfrak{f}_m$$)$, then there exist exactly 
$581$ (resp. $80$, $142$) integers $m\geq -1$ such that $h_m\leq 1000$. 
We also show that if $-1\leq m\leq 10^7$, 
then $h_m<16$ holds for $138=26+31+11+10+36+21+3$ integers $m$. 
More precisely, there exist $26$ $($resp. $31$, $11$, $10$, $36$, $21$, $3$$)$ 
integers $m$ with $-1\leq m\leq 10^7$ such that $h_m=1$ 
$($resp. $3$, $4$, $7$, $9$, $12$, $13$$)$ which are given explicitly. 
All computations of the class numbers and class groups of the listed fields 
are certified unconditionally. 
Under the GRH $($Generalized Riemann Hypothesis$)$, we confirm that 
these are the only such integers $m$ in the range $-1\leq m\leq 10^7$. 
We conjecture that the $26$ integers $m$ obtained above with $h_m=1$ 
are the only integers $m\geq -1$ with $h_m=1$. 
\end{abstract}
\maketitle
\tableofcontents

%%%%%%%%%%%%%%%%%%%%%%%%%%%%%%%%%%%%%%%%%%%%%%%%%%%%%%%%%%%%%%%%%%%%%%%%%%%%%%%
\section{Introduction}\label{S1}

Let $m\in\bZ$ be an integer and 
\begin{align*}
f_m(X)&=X^3-mX^2-(m+3)X-1\in\bZ[X] 
\end{align*}
%be the simplest cubic polynomial 
with discriminant 
\begin{align*}
{\rm disc}(f_m(X))=(m^2+3m+9)^2. 
\end{align*}
The polynomial $f_m(X)$ is irreducible over $\bQ$ for any $m\in\bZ$ 
and the splitting field $L_m=\bQ(\alpha)$ of $f_m(X)$ over $\bQ$ 
is called {\it the simplest cubic field} 
where $\alpha$ is a root of $f_m(X)$ with Galois action
\begin{align*}
\sigma: \alpha\mapsto \frac{-1}{\alpha+1}\mapsto 
-\frac{\alpha+1}{\alpha}\mapsto \alpha
\end{align*}
and ${\rm Gal}(L_m/\bQ)=\langle\sigma\rangle\simeq C_3$ 
(see Shanks \cite{Sha74}). 
A reduced form of $f_m(X)$ is given by 
\begin{align*}
3^3f_m((X+m)/3)=X^3-3(m^2+3m+9)X-(2m+3)(m^2+3m+9).
\end{align*}
Let $L_m=\bQ(\alpha)$ be the simplest cubic field 
with conductor $\mathfrak{f}_m$, 
i.e. the smallest integer $\mathfrak{f}_m\geq 1$ 
with $L_m\subset \bQ(\zeta_{\mathfrak{f}_m})$ 
where $\zeta_n$ is the primitive $n$-th root of unity, 
and $\mathcal{O}_{L_m}$ be the ring of integers of $L_m$. 
Then we have 
\begin{align*}
{\rm disc}(f_m(X))&=d_m\,[\mathcal{O}_{L_m}:\bZ[\alpha]]^2,\\
m^2+3m+9&=\mathfrak{f}_m\,[\mathcal{O}_{L_m}:\bZ[\alpha]]
\end{align*}
where $d_m=(\mathfrak{f}_m)^2$ 
is the discriminant of $L_m$ 
provided by the conductor-discriminant formula 
(see e.g. Washington \cite[Theorem 3.11, page 28]{Was97}). 
%where $\mathfrak{f}_m$ is the conductor of $L_m$. 
Because $f_m(X)=-X^3 f_{-m-3}(1/X)$, we get $L_m=L_{-m-3}$ 
for any $m\in\bZ$. 
Hence we may assume that $m\geq -1$. 
Note that (i) $m^2+3m+9=((2m+3)^2+27)/4$ is invariant under the action 
$\tau : m\mapsto -m-3$ because $\tau(2m+3)=-(2m+3)$ and 
(ii) $N_{L_m/\bQ}(\alpha)=1$, $N_{L_m/\bQ}(\alpha+1)=-1$, 
$N_{L_m/\bQ}(\alpha-1)=2m+3$. 
The simplest cubic fields $L_m$ are investigated by many authors, 
see e.g. Shanks \cite{Sha74},  Cusick \cite{Cus83}, 
Lettl \cite{Let86}, Washington \cite{Was87}, Hoshi \cite{Hos11} 
and more recent papers of 
Kashio and Sekigawa \cite{KS21}, 
Hashimoto and Aoki \cite{HA24}, 
Ogawa and Aoki \cite{OA25}, 
Komatsu \cite{Kom25}. 

Write $m^2+3m+9=3^{r_3}p_1^{e_1}\cdots p_r^{e_r}$ 
for distinct primes $p_i\geq 5$ $(1\leq i\leq r)$. 
Then we see that 
\begin{align*}
r_3=
\begin{cases}
0 & (m \not\equiv 0\ ({\rm mod}\ 3))\\
2 & (m\equiv 0, 6\ ({\rm mod}\ 9))\\
3 & (m\equiv 3\ ({\rm mod}\ 9))
\end{cases}
\end{align*}
and $p_i\equiv 1\ ({\rm mod}\ 6)$ $(1\leq i\leq r)$ 
(see e.g. Hasse \cite{Has48}, Gras \cite{Gra75}). 

\begin{theorem}[{Gras \cite{Gra73}, \cite[Proposition 2]{Gra74}, \cite[Proposition 3.10]{Gra86}, Cusick \cite[Lemma 1]{Cus83} ($m^2+3m+9$ is squarefree), Washington \cite[Proposition 1]{Was87} $(m\not\equiv 3\,({\rm mod}\ 9))$, Komatsu \cite[Theorem 3.6]{Kom04}, \cite[Lemma 1.4]{Kom07}, H\"aberle \cite[Theorem 10]{Hab10}}]\label{thG}
Let $m\in\bZ$ %$m\geq -1$ 
be an integer and $L_m=\bQ(\alpha)$ be the simplest cubic field 
where $\alpha$ is a root of $f_m(X)=X^3-mX^2-(m+3)X-1$. 
Write $m^2+3m+9=bc^3$ for $b,c\geq 0$ with cubefree integer $b$. 
Then the conductor $\mathfrak{f}_m$ of $L_m$ is given by 
\begin{align*}
\mathfrak{f}_m=\gamma\prod_{p\mid b,\, p\not\neq 3}p
\end{align*}
where 
\begin{align*}
\gamma=
\begin{cases}1& (m\not\equiv 0\ ({\rm mod}\ 3)\ 
{\rm or}\ m\equiv 12\ ({\rm mod}\ 27))\\
9& (m\equiv 0\ ({\rm mod}\ 3)\ 
{\rm and}\ m\not\equiv 12\ ({\rm mod}\ 27)). 
\end{cases}
\end{align*}
In particular, $\mathfrak{f}_m=p$ is prime if and only if 
{\rm (i)} $m\not\equiv 0\pmod{3}\ {\rm or}\ m\equiv 12\ ({\rm mod}\ 27)$ 
and {\rm (ii)} $b=p$ or $p^2$ where $p$ is prime. 
\end{theorem}

\begin{corollary}\label{corG}
%Let $f_m(X)=X^3-mX^2-(m+3)X-1$ be the simplest cubic polynomial
%Let $m\geq -1$ be an integer. 
Let $m\in\bZ$ %$m\geq -1$ 
be an integer and $L_m=\bQ(\alpha)$ be the simplest cubic field 
with conductor $\mathfrak{f}_m$ where 
$\alpha$ is a root of $f_m(X)=X^3-mX^2-(m+3)X-1$.  
Let $\mathcal{O}_{L_m}$ be the ring of integers of $L_m$. 
Write $m^2+3m+9=3^{r_3}p_1^{e_1}\cdots p_r^{e_r}$ 
for distinct primes $p_i\equiv 1\pmod{6}$ %$p_i\geq 5$ 
$(1\leq i\leq r)$ where $r_3=0$ $($resp. $2$, $3$$)$ 
if $m \not\equiv 0\ ({\rm mod}\ 3)$ 
$($resp. $m\equiv 0, 6\ ({\rm mod}\ 9)$, 
$m\equiv 3\ ({\rm mod}\ 9))$. 
Then we have 
\begin{align*}
(m^2+3m+9)/\mathfrak{f}_m=
[\mathcal{O}_{L_m}:\bZ[\alpha]]=3^kp_1^{t_1}\cdots p_r^{t_r}
\end{align*}
where 
\begin{align*}
k=
\begin{cases}
0 & (m \not\equiv 0\ ({\rm mod}\ 3)\ {\rm or}\ m\equiv 0, 6\ ({\rm mod}\ 9))\\
1 & (m\equiv 3, 21\ ({\rm mod}\ 27))\\
3 & (m\equiv 12\ ({\rm mod}\ 27))
\end{cases},\quad 
t_i=
\begin{cases}
e_i-1 & (e_i \not\equiv 0\ ({\rm mod}\ 3))\\
e_i & (e_i\equiv 0\ ({\rm mod}\ 3)).
\end{cases}
\end{align*}
In particular, we have\\
{\rm (1)} 
$[\mathcal{O}_{L_m}:\bZ[\alpha]]=1$, i.e. $f_m(X)$ is monogenic, 
if and only if 
$m^2+3m+9=\mathfrak{f}_m$ 
if and only if 
\begin{align*}
m^2+3m+9=
\begin{cases}
9\ (m=0)\\
p_1\cdots p_r\ (m \not\equiv 0\ ({\rm mod}\ 3)\ 
{\rm and}\ m^2+3m+9\ {\rm is\ squarefree})\\
9p_1\cdots p_{r-1}\ (m\equiv 0, 6\ ({\rm mod}\ 9)\ {\rm and}\ (m^2+3m+9)/9\ 
{\rm is\ squarefree});
\end{cases}
\end{align*}
{\rm (2)} 
$[\mathcal{O}_{L_m}:\bZ[\alpha]]=3$ 
if and only if 
$m^2+3m+9=3\,\mathfrak{f}_m$ 
if and only if 
\begin{align*}
m^2+3m+9=
\begin{cases}
27\ (m=3)\\
27p_1\cdots p_{r-1}\ (m \equiv 3, 21\ ({\rm mod}\ 27)\ {\rm and}\ (m^2+3m+9)/27\ 
{\rm is\ squarefree});
\end{cases}
\end{align*}
{\rm (3)} 
$[\mathcal{O}_{L_m}:\bZ[\alpha]]=27$ 
if and only if 
$m^2+3m+9=27\,\mathfrak{f}_m$ 
if and only if 
\begin{align*}
m^2+3m+9=27p_1\cdots p_{r-1}\ (m \equiv 12\ ({\rm mod}\ 27)\ {\rm and}\ (m^2+3m+9)/27\ 
{\rm is\ squarefree}).
\end{align*}
Moreover, if $[\mathcal{O}_{L_m}:\bZ[\alpha]]=1$, then 
the unit group $\mathcal{O}_{L_m}^\times=\bZ[\alpha]^\times$ 
of $L_m$ is generated by $-1$, $\alpha$, $\alpha+1$; 
$\mathcal{O}_{L_m}^\times=\langle -1,\alpha,\alpha+1\rangle$ 
$($see Shanks \cite[page 1138]{Sha74}, Thomas \cite[Theorem 3.10]{Tho79}, 
Cusick \cite[Theorem 1]{Cus83}, 
Washington \cite[Corollary, page 372]{Was87}, 
see also Conrad \cite[Theorem 5.12]{Con}$)$. 
\end{corollary}

There exist $2^{r-1}$ (resp. $1$) 
cyclic cubic fields with conductor $\mathfrak{f}_m=p_1\cdots p_r$ or 
$\mathfrak{f}_m=9p_1\cdots p_{r-1}$ (resp. $\mathfrak{f}_m=9$). 
Gras \cite[Table 1]{Gra75} 
(resp. \cite[Table 2, Table 3]{Gra75}) 
gives the class numbers $h$ of the cyclic cubic fields of $\bQ$ 
with conductor $\mathfrak{f}\leq 4000$ 
(resp. $4000<\mathfrak{f}<10000$ and $h\equiv 0\pmod{9}$, 
$4000<\mathfrak{f}<10000$ and $h\equiv 0\pmod{4}$). 
Note that if we take $m\in\bQ$, then $L_m$ 
gives all the cyclic cubic fields of $\bQ$ 
(see Serre \cite[Section 1.1]{Ser92}, 
Jensen, Ledet and Yui \cite[Section 2.1]{JLY02}) but, for our case $m\in\bZ$, 
$L_m$ gives only the special cyclic cubic fields (for example, 
$2$ should be inert in $L_m$ because $f_m(X)$ is irreducible 
over $\bZ/2\bZ$). 

Let $\Cl(L_m)$ be the ideal class group of $L_m$. 
Then the class number $h_m=|\Cl(L_m)|$ of $L_m$
is of the form $h_m=A^2+3B^2$, in particular, $h_m\equiv 0, 1\pmod{3}$ 
(see e.g. Hasse \cite{Has48}, 
Shanks \cite[page 1145]{Sha74}). 
%In particular, $h_m\equiv 0,1\pmod{3}$. 
%
\begin{theorem}[{Shanks \cite[Theorem, page 1144]{Sha74}}]
Let $m\in\bZ$ be an integer and 
$L_m=\bQ(\alpha)$ be the simplest cubic field 
with ideal class group $\Cl(L_m)$ 
where $\alpha$ is a root of $f_m(X)=X^3-mX^2-(m+3)X-1$. 
Let $x\in \Cl(L_m)$ be an element of order $s$. 
If $s$ is divisible only by primes $\equiv 2\pmod{3}$, then 
the $s$-rank of $\Cl(L_m)$ must be even. 
\end{theorem}
For example, if $s=p^{2n+1}$, $p\equiv 2\pmod{3}$ is prime, 
$\gcd(s,t)=1$, then $h_m=st$ does not occur. 
In particular, possible class numbers $h_m=|\Cl(L_m)|$ of $L_m$ 
with $h_m<43$ are $h_m=1,3,4,7,9,12,13,16,19,21,25,27,28,31,36,37,39$. 
\begin{theorem}[{see e.g. Gras \cite[Section 3, page 94]{Gra75}, 
see also Lemmermeyer \cite[Theorem 1]{Lem13} for the ambiguous class number formula $|\Cl(L_m)^{C_3}|=3^{r-1}$}]\label{thG2}
%Let $L$ be a cubic field of $\bQ$ with conductor $\mathfrak{f}$. 
Let $m\in\bZ$ be an integer and 
$L_m=\bQ(\alpha)$ be the simplest cubic field with 
class number $h_m$ and conductor $\mathfrak{f}_m$  
where $\alpha$ is a root of $f_m(X)=X^3-mX^2-(m+3)X-1$. 
Then $\mathfrak{f}_m=9$ or 
$\mathfrak{f}_m=p_1\cdots p_r$ or $\mathfrak{f}_m=9p_1\cdots p_{r-1}$ 
for distinct primes $p_i\equiv 1\pmod{6}$ $(1\leq i\leq r)$ and we have:\\  
{\rm (i)} $h_m\equiv 1\pmod{3}$ if and only if $\Cl(L_m)^{C_3}=\{1\}$ 
if and only if $\mathfrak{f}_m=9$ or $\mathfrak{f}_m=p_1$ is prime with 
$p_1 \equiv 1\pmod{6}$;\\
{\rm (ii)} $3^{r-1}$ divides $h_m$ 
if $\mathfrak{f}_m=p_1\cdots p_r$ or 
$\mathfrak{f}_m=9p_1\cdots p_{r-1}$ with $r\geq 2$.
\end{theorem}

Note that $L_m$ $(m\geq -1)$ with $\mathfrak{f}_m=9$ if and only if 
$m=0,3,54$ (see Theorem \ref{thH4}). 
Because $L_0=L_3=L_{54}=\bQ(\zeta_9+\zeta_9^{-1})$, 
we have $h_0=h_3=h_{54}=1$. 
Hence if $h_m\equiv 1\pmod{3}$, e.g. $h_m=1$, 
then {\rm (i)} $m=0,3,54$ (with $h_m=1$) or 
{\rm (ii)-1} $m\not\equiv 0\pmod{3}\ {\rm or}\ m\equiv 12\ ({\rm mod}\ 27)$ 
and {\rm (ii)-2} $m^2+3m+9=pc^3$ or $p^2c^3$ 
where $p\equiv 1\pmod{6}$ is prime and $c\geq 1$. 

Lettl \cite[page 660]{Let86} established 
a lower bound for $L(1,\chi)L(1,\overline{\chi})$ 
and hence got a lower bound of the class number $h_m$ 
of the simplest cubic field $L_m$ 
with the aid of the class number formula 
(see e.g. Washington \cite[Chapter 4]{Was97}) 
and an upper bound of the regulator
$R_m<(\log (m^2+3m+9))^2/4$ of $L_m$ (see Lettl \cite[Lemma 1]{Let86}): 
\begin{theorem}[{Lettl \cite[Section 2, pages 660--661]{Let86}}]\label{thLettl}
Let $m\in\bZ$ be an integer and $L_m=\bQ(\alpha)$ be the simplest cubic field 
with class number $h_m$, conductor $\mathfrak{f}_m$ and regulator $R_m$ 
where $\alpha$ is a root of $f_m(X)=X^3-mX^2-(m+3)X-1$. 
Let $\mathcal{O}_{L_m}$ be the ring of integers of $L_m$. 
Let $L(s,\chi)$ and $L(s,\overline{\chi})$ be the Dirichlet $L$-functions 
associated to the nontrivial cubic Dirichlet characters 
$\chi$ and $\overline{\chi}$ modulo $\mathfrak{f}_m$. 
If $\mathfrak{f}_m>10^5$, then 
\begin{align*}
|L(1,\chi)|^2=L(1,\chi)L(1,\overline{\chi})>0.023\,(\mathfrak{f}_m)^{-0.054}.
\end{align*}
Hence, by the class number formula
\begin{align*}
h_m=\frac{\mathfrak{f}_m\,|L(1,\chi)|^2}{4R_m}
\end{align*}
and the upper bound of the regulator $R_m<(\log (m^2+3m+9))^2/4$, 
we get that if $\mathfrak{f}_m>10^5$, then 
\begin{align*}
h_m>0.023\,\frac{(\mathfrak{f}_m)^{0.946}}{(\log (m^2+3m+9))^2}
\end{align*}
where $m^2+3m+9=\mathfrak{f}_m\,[\mathcal{O}_{L_m}:\bZ[\alpha]]$. 
In particular, for the fixed index $[\mathcal{O}_{L_m}:\bZ[\alpha]]$ 
and the fixed integer $h\geq 1$, 
there exist only finitely many integers $m\in\bZ$ such that $h_m=h$. 
For example, if $[\mathcal{O}_{L_m}:\bZ[\alpha]]=1$ $($resp. $3$, $27$$)$, 
i.e. $m^2+3m+9=\mathfrak{f}_m$ 
$($resp. $3\mathfrak{f}_m$, $27\mathfrak{f}_m$$)$, and $m\geq 410$ 
$($resp. $794$, $2870$$)$, 
i.e. $m^2+3m+9\geq 169339$ $($resp. $632827$, $8245519$$)$, then $h_m>14$. 
\end{theorem}

By using Theorem \ref{thLettl}, Lettl \cite{Let86} got: 
\begin{theorem}[{Lettl \cite[Theorem, page 661]{Let86}}]\label{thLet86}
Let $m\geq -1$ be an integer and $L_m=\bQ(\alpha)$ 
be the simplest cubic field 
with class number $h_m$ and conductor $\mathfrak{f}_m$ 
where $\alpha$ is a root of $f_m(X)=X^3-mX^2-(m+3)X-1$. 
Let $\mathcal{O}_{L_m}$ be the ring of integers of $L_m$. 
Assume that $\mathfrak{f}_m=p$ is prime. Then\\ 
{\rm (1)} If $[\mathcal{O}_{L_m}:\bZ[\alpha]]=1$, i.e. 
$m^2+3m+9=\mathfrak{f}_m$, then 
$h_m<16$ holds only for the $19$ integers $m\geq -1$. 
More precisely, there exist exactly 
$7$ $($resp. $5, 6, 1)$ integers $m=-1,1,2,4,7,8,10$ 
$($resp. $m=11,17,23,25,29$, $m=16,28,32,38,43,49$, $m=31$$)$ 
such that $h_m=1$ $($resp. $4$, $7$, $13$$)$.\\
{\rm (2)} If $[\mathcal{O}_{L_m}:\bZ[\alpha]]=27$, i.e. 
$m^2+3m+9=27\mathfrak{f}_m$, then 
$h_m<43$ holds only for the $12$ integers $m\geq -1$. 
More precisely, there exist exactly 
$3$ $($resp. $2, 2, 1, 2, 1, 1)$ integers 
$m=12,39,93$ $($resp. $m=120,228$, $m=255,309$, $m=498$, 
$m=336,822$, $m=795$, $m=471$$)$ 
such that $h_m=1$ 
$($resp. $4$, $7$, $13$, $28$, $31$, $37$$)$. 
\end{theorem}

\begin{remark}\label{r1.7}
(1) The case $m=12$ with $m^2+3m+9=189=27\mathfrak{f}_m$, 
$\mathfrak{f}_m=7$ and $h_m=1$ as in Theorem \ref{thLet86} (2) 
is missing in Lettl \cite[Theorem, page 661]{Let86} 
(see also Table $1$).\\
(2) Lettl \cite[Section 2, page 660]{Let86} called 
the simplest cubic field $L_m=\bQ(\alpha)$ with 
$[\mathcal{O}_{L_m}:\bZ[\alpha]]=1$ (resp. $27$) 
{\it of Type $A$} (resp. {\it of Type $B$}). 
\end{remark}

Similarly, by using Theorem \ref{thLettl}, 
Byeon \cite[Theorem 3.3]{Bye00} checked that 
if $[\mathcal{O}_{L_m}:\bZ[\alpha]]=1$, i.e. $m^2+3m+9=\mathfrak{f}_m$, 
then %$h_m=3$ holds only for the following integers of $m\geq -1$; 
there exist exactly $9=5+4$ integers $m\geq -1$ such that $h_m=3$:\\
{\rm (1)} If $m\not\equiv 0\ ({\rm mod}\ 3)$ and $\mathfrak{f}_m=p_1p_2$ 
where $p_1$, $p_2\equiv 1\pmod{6}$ are distinct primes, 
then there exist exactly $5$ integers $m=13,14,19,20,22$ 
$(\mathfrak{f}_m=7\cdot 31, 13\cdot 19, 7\cdot 61, 7\cdot 67, 13\cdot 43)$ 
such that $h_m=3$;\\
{\rm (2)} If $m\equiv 0,6\ ({\rm mod}\ 9)$ and $\mathfrak{f}_m=9p$ 
where $p\equiv 1\pmod{6}$ is prime, 
then there exist exactly $4$ integers $m=6,9,15,18$ $(p=7,13,31,43)$ 
such that $h_m=3$. 
%\end{theorem}
%

Louboutin \cite{Lou02} generalizes 
results of Lettl \cite{Let86} and Byeon \cite{Bye00} 
(cf. Theorem \ref{thLettl}): 
\begin{theorem}[{Louboutin \cite[Lemma 3, see also Theorem 4]{Lou02}}]\label{thLou}
Let $m\in\bZ$ be an integer and $L_m=\bQ(\alpha)$ be the simplest cubic field 
with class number $h_m$, conductor $\mathfrak{f}_m$ and regulator $R_m$ 
where $\alpha$ is a root of $f_m(X)=X^3-mX^2-(m+3)X-1$. 
Let $\mathcal{O}_{L_m}$ be the ring of integers of $L_m$. 
Let $\zeta_{L_m}(s)$ be the Dedekind zeta function of $L_m$ and 
$L(s,\chi)$ %and $L(s,\overline{\chi})$ 
be the Dirichlet $L$-function associated to a nontrivial 
cubic Dirichlet character 
$\chi$ %and $\overline{\chi}$ 
modulo $\mathfrak{f}_m$. 
If $\mathfrak{f}_m>2\sqrt{3}\cdot 10^4$, then 
\begin{align*}
|L(1,\chi)|^2={\rm Res}_{s=1}(\zeta_{L_m})>\frac{1}{e\,\log\mathfrak{f}_m}.
\end{align*}
Hence, by the class number formula
\begin{align*}
h_m=\frac{\mathfrak{f}_m\,|L(1,\chi)|^2}{4R_m}
\end{align*}
and the upper bound of the regulator $R_m<(\log (m^2+3m+9))^2/4$, 
we get that if $\mathfrak{f}_m>2\sqrt{3}\cdot 10^4$, then 
\begin{align*}
h_m>\frac{\mathfrak{f}_m}{e\,(\log (m^2+3m+9))^2\log \mathfrak{f}_m}
\end{align*}
where $m^2+3m+9=\mathfrak{f}_m\,[\mathcal{O}_{L_m}:\bZ[\alpha]]$. 
In particular, for the fixed index $[\mathcal{O}_{L_m}:\bZ[\alpha]]$ 
and the fixed integer $h\geq 1$, 
there exist only finitely many integers $m\in\bZ$ such that $h_m=h$. 
For example, if $[\mathcal{O}_{L_m}:\bZ[\alpha]]=1$ $($resp. $3$, $27$$)$, 
i.e. $m^2+3m+9=\mathfrak{f}_m$ 
$($resp. $3\mathfrak{f}_m$, $27\mathfrak{f}_m$$)$, and $m\geq 217$ 
$($resp. $429$, $1600$$)$, 
i.e. $m^2+3m+9\geq 47749$ $($resp. $185337$, $2564809$$)$, then $h_m>14$. 
\end{theorem}

%%%%%%%%%%%%%%%%%%%%%%%%%%%%%%%%%%%%%%%%%%%%%%%
A celebrated result of Thue \cite{Thu09} claims 
that an equation $F(x,y)=\lambda$ 
where $F(X,Y)\in\bZ[X,Y]$ is an irreducible binary form of degree $d\geq 3$ 
and $\lambda\in\bZ\setminus\{0\}$ 
has only finitely many integral 
solutions $(x,y)\in\bZ^2$. 
Baker \cite{Bak68} proved that the equation $F(x,y)=\lambda$ 
can be solved effectively. 
Numerical methods for solving a Thue equation are developed by Tzanakis and 
de Weger \cite{TW89} and Bilu and Hanrot \cite{BH96}. 
Thomas \cite{Tho90} investigated %for the first time 
the family of cubic Thue equations 
$F_m(X,Y)=\pm 1$ where $m\in\bZ$ and 
\[
F_m(X,Y)=Y^3f_m(X/Y)=X^3-mX^2Y-(m+3)XY^2-Y^3.
\]
In the case where $\lambda=\pm 1$, Thomas \cite{Tho90} showed that for 
$m+1\geq 1.365\cdot 10^7$, the equation $F_m(x,y)=\pm 1$ has only the trivial 
solutions $(x,y)=(\pm 1,0)$, $(0,\mp 1)$, $(\mp 1,\pm 1)$ and for 
$0\leq m+1\leq 10^3$, nontrivial solutions exist only for $3$ integers 
$m=-1$, $0$, $2$. 
More precisely, 
there exist $6$ (resp. $3$, $3$) nontrivial solutions 
for $m=-1$ (resp. $m=0$, $m=2$). 
Mignotte \cite{Mig93} 
solved completely the equations $F_m(X,Y)=\pm 1$ with 
the aid of a result of Thomas \cite{Tho90}. 
In particular,  for $m\geq -1$, nontrivial solutions 
exist only for $m=-1$, $0$, $2$. 
Mignotte, Peth\"o and Lemmermeyer \cite{MPL96} 
studied the equation $F_m(X,Y)=\lambda$ for general $\lambda\in\bZ$ 
and gave complete solutions to Thue inequality 
$|F_m(X,Y)|\leq 2m+3$ 
with the aid of a result of Lemmermeyer and Peth\"o \cite{LP95}. 
In particular, %for $m\geq -1$, 
all the solutions are the trivial ones 
$(x,y)=(c,0)$, $(0,-c)$, $(-c,c)$ to $F_m(x,y)=c^3$ and 
nontrivial $6$ solutions 
$(x,y)=(-1,-1)$, $(-1,2)$, $(2,-1)$, $(-m-1,-1)$, $(-1,m+2)$, $(m+2,-m-1)$ 
to $F_m(x,y)=2m+3$ for any $m\in\bZ$ and $6$ extra solutions 
$(x,y)=(3,1)$, $(1,-4)$, $(-4,3)$, $(8,3)$, $(3,-11)$, $(-11,8)$ to 
$F_1(x,y)=2m+3=5$ with $m=1$. 
Hoshi \cite{Hos11} determined solutions to the families of cubic 
Thue equations $F_m(X,Y)=\lambda$  
where $m\in\bZ$ and $\lambda$ is a divisor of $m^2+3m+9$ 
using a solution to the field isomorphism problem 
for the simplest cubic fields as follows 
(see also Lettl, Peth\"o, Voutier \cite{LPV99}, 
Hoshi \cite{Hos14}, \cite{Hos12} for the quartic and the sextic cases): 
\begin{theorem}[{Hoshi \cite[Theorem 1.1]{Hos11}}]\label{thH1}
Let $m\in\bZ$ be an integer and $L_m=\bQ(\alpha)$ be the simplest cubic field 
where $\alpha$ is a root of $f_m(X)=X^3-mX^2-(m+3)X-1$. 
There exists $n\in\bZ\setminus\{m,-m-3\}$ such that $L_n=L_m$ 
if and only if there exists a solution $(x,y)\in\bZ^2$ 
with $xy(x+y)\neq 0$ to the family of cubic Thue equations
\begin{align*}
F_m(x,y)=x^3-mx^2y-(m+3)xy^2-y^3=\lambda
\end{align*}
where $\lambda>0$ is a divisor of $m^2+3m+9$. 
Moreover, an integer $n$ and solutions $(x,y)\in\bZ^2$ to 
$F_m(x,y)=\lambda$ can be chosen to satisfy
\begin{align*}
N=m+\frac{(m^2+3m+9)xy(x+y)}{F_m(x,y)}
\end{align*}
where either $N=n$ or $N=-n-3$ and this occurs for only one of $n$ and $-n-3$. 
\end{theorem}
\begin{theorem}[{Hoshi \cite[Theorem 1.5]{Hos11}}]\label{thH2}
Let $F_m(X,Y)=Y^3f_m(X/Y)=X^3-mX^2Y-(m+3)XY^2-Y^3$. 
For $m\geq -1$, there exist exactly $66$ solutions 
$(x,y)\in\bZ^2$ with $xy(x+y)\neq 0$ to the family of Thue equations 
$F_m(x,y)=\lambda$ where $\lambda>0$ is a divisor of $m^2+3m+9$. 
The $66$ solutions are given as in \cite[Table 1]{Hos11}. 
\end{theorem}

Ennola \cite[Table 2.1]{Enn91} verified that 
for integers $-1\leq m<n\leq 10^4$, 
the field coincidence $L_m=L_n$ of the simplest cubic fields occurs 
if and only if $(m,n)\in S=\{(-1,5)$, $(-1,12)$, $(-1,1259)$, $(5,12)$, $(5,1259)$, 
$(12,1259)\}$ $\cup$ $\{(0,3)$, $(0,54)$, $(3,54)\}$ $\cup$ $\{(1,66)\}$ $\cup$ $\{(2,2389)\}$ with $|S|=11$. 
Hoshi and Miyake \cite[Example 5.3]{HM09} 
checked with the aid of computer that this claim is also valid for 
$-1\leq m<n\leq 10^5$. 

Okazaki \cite{Oka02} investigated Thue equations $F(X,Y)=1$ for irreducible cubic forms $F$ 
with positive discriminant ${\rm disc}(F)>0$ and established a very strong result on gaps between solutions. 
Using results in Okazaki \cite{Oka02}, 
Okazaki \cite{Oka} announced the following theorems and 
actually full proofs of them are given by Hoshi \cite{Hos11}: 
\begin{theorem}[{Okazaki \cite{Oka}, Hoshi \cite[Theorem 1.3]{Hos11}}]\label{thH3}
For integers $-1\leq m<n$, if 
the field coincidence $L_m=L_n$ of the simplest cubic fields occurs, 
then $m\leq 35731$. 
\end{theorem}
\begin{theorem}[{Okazaki \cite{Oka}, Hoshi \cite[Theorem 1.4]{Hos11}}]\label{thH4}
For integers $-1\leq m<n$, if 
the field coincidence $L_m=L_n$ of the simplest cubic fields occurs, 
then 
$m,n\in\{-1$, $0$, $1$, $2$, $3$, $5$, $12$, $54$, $66$, $1259$, $2389\}$. 
Moreover, we have 
\begin{align*}
&L_{-1}=L_5=L_{12}=L_{1259}& &\hspace*{-33mm}{\rm with\ conductor}\ \mathfrak{f}_{-1}=\mathfrak{f}_5=\mathfrak{f}_{12}=\mathfrak{f}_{1259}=7,\\
&L_0=L_3=L_{54}& &\hspace*{-33mm}{\rm with\ conductor}\ \mathfrak{f}_0=\mathfrak{f}_3=\mathfrak{f}_{54}=9,\\
&L_1=L_{66}& &\hspace*{-33mm}{\rm with\ conductor}\ \mathfrak{f}_1=\mathfrak{f}_{66}=13,\\
&L_2=L_{2389}& &\hspace*{-33mm}{\rm with\ conductor}\ \mathfrak{f}_2=\mathfrak{f}_{2389}=19. 
\end{align*}
\end{theorem}

Note that for $m=-1$ (resp. $5$, $12$, $1259$, $0$, $3$, $54$, $1$, $66$, $2$, $2389$), 
we have $m^2+3m+9=7$ (resp. $7^2$, $3^3 7$, $7\cdot 61^3$, $3^2$, $3^3$, $3^2 7^3$, $13$, 
$3^3 13^2$, $19$, $19\cdot 67^3$) (see also Table $1$). 
For example, it follows from Theorem \ref{thH4} that\\
(i) the equation $x^2+3x+9=3^r$ 
has only the integral solutions $(x,r)=(0,2)$, $(3,3)$, $(-3,2)$, $(-6,3)$;\\
(ii) for primes $p=m^2+3m+9\equiv 1\pmod{6}$, 
e.g. $p=7$, $13$, $19$, $31$, $37$, $43$, $61$, $67$, $73$, $79$, $97$, 
the equation $x^2+3x+9=p^r$ $(r\geq 2)$ has only the integral solutions 
$(x,p,r)=(5,7,2)$, $(-8,7,2)$.  

Originally, Okazaki \cite{Oka} seems to get Theorem \ref{thH4} by using 
Theorem \ref{thH3} directly. 
However, Hoshi \cite{Hos11} first showed Theorem \ref{thH2} by using 
Theorem \ref{thH1} and Theorem \ref{thH3}. 
Then Theorem \ref{thH4} can be obtained 
as a consequence of Theorem \ref{thH1} and Theorem \ref{thH2}. 
Conversely, if we assume Theorem \ref{thH4}, then by applying Theorem \ref{thH1} to fixed $m$ and $n$, we can get Theorem \ref{thH2}. 
Note that the $66$ solutions as in Theorem \ref{thH2} 
correspond to the $11$ field coincidences $L_m=L_n$ for $(m,n)\in S$ 
with $|S|=11$ as above 
%$L_{-1}=L_5$, $L_{-1}=L_{12}$, $L_{-1}=L_{1259}$, 
%$L_5=L_{12}$, $L_5=L_{1259}$, $L_{12}=L_{1259}$, 
%$L_0=L_3$, $L_0=L_{54}$, $L_3=L_{54}$, 
%$L_1=L_{66}$, $L_2=L_{2389}$ of the simplest cubic fields $L_m$ 
by Theorem \ref{thH1} ($66=2\cdot 3\cdot 11$ and the first 
$2$ comes from $m\leftrightarrow n$ and the second 
$3$ comes from the Galois action $(x,y)\mapsto (y,-x-y)$ of order $3$).

The main results of this paper are given as follows: 
%\newpage 
\begin{theorem}\label{thmain1}
Let $m\geq -1$ be an integer and $L_m=\bQ(\alpha)$ 
be the simplest cubic field 
with class number $h_m$ and conductor $\mathfrak{f}_m$ 
where $\alpha$ is a root of $f_m(X)=X^3-mX^2-(m+3)X-1$. 
Let $\mathcal{O}_{L_m}$ be the ring of integers of $L_m$. 
Let $p\equiv 1\pmod{6}$ be prime and $s=p_1\cdots p_r$ 
$(r\geq 2)$ be a product of distinct primes 
$p_i\equiv 1\pmod{6}$ $(1\leq i\leq r)$.  
If $[\mathcal{O}_{L_m}:\bZ[\alpha]]=1$ $($resp. $3$, $27$$)$, 
then there exist exactly 
$581=1+149+308+56+67$ $($resp. $80=1+34+45$, $142=45+97$$)$ integers 
$m\geq -1$ such that $h_m\leq 1000$. 
These integers $m\geq -1$ are given explicitly 
as in Section \ref{S3} $($cf. Theorem \ref{thLet86}, 
a paragraph after Remark \ref{r1.7}, Table $1$ and Table $2$ 
for the cases with class number $h_m=1$ and $h_m=3$$)$:\\
{\rm (1)} When $[\mathcal{O}_{L_m}:\bZ[\alpha]]=1$, i.e. 
$m^2+3m+9=\mathfrak{f}_m$.\\
{\rm (i)} If $m=0$, then we have $L_0$ with $\mathfrak{f}_0=9$ and $h_0=1$.\\ 
{\rm (ii)} If $m\not\equiv 0$ $({\rm mod}\ 3)$ and $\mathfrak{f}_m=p$, 
then $h_m\equiv 1\pmod{3}$ and there exist exactly $149$ integers 
$m\geq -1$ such that $h_m\leq 1000$. 
In particular, there exist exactly $7$ %$($resp. $5$, $6$, $1$$)$ 
integers $m=-1,1,2,4,7,8,10$ 
%$($resp. $m=11,17,23,25,29$, $m=16,28,32,38,43,49$, $m=31$$)$ 
such that $h_m=1$;\\ %$($resp. $4$, $7$, $13$$)$;\\
{\rm (iii)} If $m\not\equiv 0$ $({\rm mod}\ 3)$ and $\mathfrak{f}_m=s=p_1\cdots p_r$ $(r\geq 2)$, 
then $h_m\equiv 0\pmod{3^{r-1}}$ and there exist exactly $308$ integers $m\geq -1$ such that
$h_m\leq 1000$. 
In particular, there exist exactly $5$ integers $m=13,14,19,20,22$ 
$(\mathfrak{f}_m=7\cdot 31, 13\cdot 19, 7\cdot 61, 7\cdot 67, 13\cdot 43)$ 
such that $h_m=3$;\\
{\rm (iv)} If $m\equiv 0$, $6\pmod{9}$ and $\mathfrak{f}_m=9p$, 
then $h_m\equiv 0\pmod{3}$ and there exist exactly $56$ integers $m\geq -1$ such that $h_m\leq 1000$. 
In particular, there exist exactly $4$ integers $m=6,9,15,18$ $(p=7,13,31,43)$ 
such that $h_m=3$;\\
{\rm (v)} If $m\equiv 0$, $6\pmod{9}$ and $\mathfrak{f}_m=9s=9p_1\cdots p_r$ $(r\geq 2)$, 
then $h_m\equiv 0\pmod{3^r}$ and there exist exactly $67$ integers $m\geq -1$ such that $h_m\leq 1000$.\\
{\rm (2)} When $[\mathcal{O}_{L_m}:\bZ[\alpha]]=3$, i.e. 
$m^2+3m+9=3\mathfrak{f}_m$.\\
{\rm (i)} If $m=3$, then we have $L_3$ with $\mathfrak{f}_3=9$ and $h_3=1$.\\ 
{\rm (ii)} If $m\equiv 3$, $21\pmod{27}$ and $\mathfrak{f}_m=9p$, 
then $h_m\equiv 0\pmod{3}$ and there exist exactly $34$ integers $m\geq -1$ such that $h_m\leq 1000$. 
In particular, there exist exactly $2$ integers $m=21,30$ 
$(p=19,37)$ such that $h_m=3$;\\
{\rm (iii)} If $m\equiv 3$, $21\pmod{27}$ and $\mathfrak{f}_m=9s=9p_1\cdots p_r$ $(r\geq 2)$, 
then $h_m\equiv 0\pmod{3^r}$ and there exist exactly $45$ integers $m\geq -1$ such that $h_m\leq 1000$.\\
{\rm (3)} When $[\mathcal{O}_{L_m}:\bZ[\alpha]]=27$, i.e. 
$m^2+3m+9=27\mathfrak{f}_m$.\\
{\rm (i)} If $m\equiv 12\pmod{27}$ and $\mathfrak{f}_m=p$, 
 then $h_m\equiv 1\pmod{3}$ and there exist exactly $45$ integers $m\geq -1$ such that $h_m\leq 1000$. 
In particular, there exist exactly 
$3$ %$($resp. $2$, $2$, $1$, $2$, $1$, $1$$)$ 
integers $m=12$, $39$, $93$ 
%$($resp. $m=120,228$, $m=255,309$, $m=498$, $m=336,822$, $m=795$, $m=471$$)$ 
such that $h_m=1$;\\ %$($resp. $4$, $7$, $13$, $28$, $31$, $37$$)$;\\
{\rm (ii)} If $m\equiv 12\pmod{27}$ and $\mathfrak{f}_m=s=p_1\cdots p_r$ $(r\geq 2)$, 
then $h_m\equiv 0\pmod{3^{r-1}}$ and there exist exactly $97$ integers $m\geq -1$ such that $h_m\leq 1000$. 
In particular, there exist exactly $2$ integers $m=147,174$ 
$(\mathfrak{f}_m=19\cdot 43, 7\cdot 163)$ such that $h_m=3$. 
\end{theorem}

\begin{theorem}\label{thmain2}
Let $-1\leq m\leq 10^7$ be an integer and $L_m=\bQ(\alpha)$ 
be the simplest cubic field 
with class number $h_m$ %and conductor $\mathfrak{f}_m$ 
where $\alpha$ is a root of $f_m(X)=X^3-mX^2-(m+3)X-1$. 
Then $h_m<16$ holds for $138=26+31+11+10+36+21+3$ 
integers $m$ with $-1\leq m\leq 10^7$. 
More precisely, 
there exist $26$ $($resp. $31$, $11$, $10$, $36$, $21$, $3$$)$ 
integers $m$ with $-1\leq m\leq 10^7$ such that $h_m=1$ 
$($resp. $3$, $4$, $7$, $9$, $12$, $13$$)$ 
as in Table $1$ $($resp. $2$, $3$, $4$, $5$, $6$, $7$$)$. 
Moreover, under the GRH $($Generalized Riemann Hypothesis$)$, 
$h_m<16$ holds only for such $138$ integers $m$ 
with $-1\leq m\leq 10^7$. 
\end{theorem}

\begin{remark}\label{rem12}
(1) We get the following table as a summary of 
Theorem \ref{thmain2} (Tables $1$--$7$):\vspace*{2mm}
\begin{center}
\begin{tabular}{l|rrrrrrr}
$|\Cl(L_m)|=h_m<16$ & 1 & 3 & 4 & 7 & 9 & 12 & 13\\\hline
$\#\{-1\leq m\leq 10^7\mid |\Cl(L_m)|=h_m\}$ & 
26 & 31& 11 & 10 & 36 & 21 & 3\\
{\rm The largest} $m\leq 10^7=10000000$ & 506370 & 1376233 & 6440111 & 612049 & 
1360624 & 14349 & 36435
\end{tabular}\vspace*{2mm}
\end{center}
(2) By Theorem \ref{thG2}, if class number $h_m=1$, then 
conductor $\mathfrak f_m=9$ or 
$\mathfrak f_m=p$ where $p\equiv 1\pmod{6}$ is prime. 
Moreover, in view of Theorem \ref{thH4}, the corresponding $26$ integers 
$m\geq -1$ as in Table $1$ give, up to $\mathbb Q$-isomorphism,  
$19$ simplest cubic fields $L_m$ 
with class number $h_m=1$  
(see also Theorem \ref{thG} and Corollary \ref{corG} for the conductor 
$\mathfrak{f}_m$ of $L_m$ 
and Table $1$ for the prime factorization of $m^2+3m+9$):\vspace*{2mm}
\begin{center}
\begin{tabular}{l|rrrrrrrrrr}
$m\ {\rm with}\ h_m=1$ & $-1,5,12,1259$ & $0,3,54$ & $1, 66$ & $2, 2389$ & 
$4$ & $7$ & $8$ & $10$ & $39$ & $93$ \\\hline
$\mathfrak{f}_m$ & 7 & 9 & 13 & 19 & 37 & 79 & 97 & 
139 & 61 & 331\vspace*{2mm}\\\
%%%%%
$m\ {\rm with}\ h_m=1$ & $286$ & $397$ & $911$ & $1283$ & 
$1598$ & $7837$ & $12745$ & $263135$ & $506370$\\\hline
$\mathfrak{f}_m$ & 241 & 463 & 379 & 751 & 373 & 
1213 & 1381 & 2539 & 1321\vspace*{2mm} 
\end{tabular}
\end{center}
\end{remark}

These observations, together with Theorem \ref{thmain2}, 
give rise to the following conjecture: 

\begin{conjecture}\label{conj1}
The $26$ integers $m$ as in Table $1$ are the only integers $m\geq-1$ 
with class number $h_m=1$. 
Namely, in view of Theorem \ref{thH4}, up to $\mathbb Q$-isomorphism, 
the $19$ cubic fields $L_m$ with conductor 
\begin{align*}
\mathfrak{f}_m=7,\, 9,\, 13,\, 19,\, 37,\, 61,\, 79,\, 97,\, 139,\, 
241,\, 331,\, 373,\, 379,\, 463,\, 751,\, 1213,\, 1321,\, 1381,\, 2539
\end{align*}
as in Remark \ref{rem12} (2) are exactly the simplest cubic fields $L_m$ $(m\in\bZ)$ 
with class number $h_m=1$. 
\end{conjecture}

\begin{remark}
Although Conjecture \ref{conj1} claims only the class number one case, 
the computations in Theorem \ref{thmain2} suggest that analogous 
statements may also hold for $h_m=3$, $4$, $7$, $9$, $12$, $13$. 
More precisely, it would be interesting to determine whether, 
for $h\in\{3$, $4$, $7$, $9$, $12$, $13\}$, 
Tables $2$--$7$ contain all integers $m\geq -1$ with $h_m=h$. 
\end{remark}

\begin{acknowledgment}
We would like to thank St\'ephane Louboutin 
for helpful explanations concerning Theorem \ref{thLou} 
which improved the proof of Theorem \ref{thmain1}. 
We also thank him for informing us about the 
related and interesting paper Louboutin \cite[Section 3]{Lou20}. 
\end{acknowledgment}

\newpage
{
\begin{multicols}{2}
\begin{center}
Table $1$: $-1\leq m\leq 10^7$ with class number $|\Cl(L_m)|=1$
\begin{tabular}{rll}
$m$ & $m\pmod{3,9,27}$ & $m^2+3m+9$\\\hline
$-1$ & $\not\equiv 0\pmod{3}$ & $7$\\
$0$ & $\equiv 0\pmod{9}$ & $9=3^2$\\ %& $9=3^2$\\
$1$ & $\not\equiv 0\pmod{3}$ & $13$\\
$2$ & $\not\equiv 0\pmod{3}$ & $19$\\
$3$ & $\equiv 3\pmod{27}$ & $27=3^3$\\ %& $27=3^3$\\
$4$ & $\not\equiv 0\pmod{3}$ & $37$\\
$5$ & $\not\equiv 0\pmod{3}$ & $7^2$\\ %& $49=7^2$\\
$7$ & $\not\equiv 0\pmod{3}$ & $79$\\
$8$ & $\not\equiv 0\pmod{3}$ & $97$\\
$10$ & $\not\equiv 0\pmod{3}$ & $139$\\
$12$ & $\equiv 12\pmod{27}$ & $3^3 7$\\ %& $189=3^3 7$\\
$39$ & $\equiv 12\pmod{27}$ & $3^3 61$\\ %& $1647=3^3 61$\\
$54$ & $\equiv 0\pmod{9}$ & $3^2 7^3$\\ %& $3087=3^2 7^3$\\
$66$ & $\equiv 12\pmod{27}$ & $3^3 13^2$\\ %& $4563=3^3 13^2$\\
$93$ & $\equiv 12\pmod{27}$ & $3^3 331$\\ %& $8937=3^3 331$\\
$286$ & $\not\equiv 0\pmod{3}$ & $7^3 241$\\ %& $82663=7^3 241$\\
$397$ & $\not\equiv 0\pmod{3}$ & $7^3 463$\\ %& $158809=7^3 463$\\
$911$ & $\not\equiv 0\pmod{3}$ & $13^3 379$\\ %& $832663=13^3 379$\\
$1259$ & $\not\equiv 0\pmod{3}$ & $7\cdot 61^3$\\ %& $1588867=7\cdot 61^3$\\
$1283$ & $\not\equiv 0\pmod{3}$ & $13^3 751$\\ %& $1649947=13^3 751$\\
$1598$ & $\not\equiv 0\pmod{3}$ & $19^3 373$\\ %& $2558407=19^3 373$\\
$2389$ & $\not\equiv 0\pmod{3}$ & $19\cdot 67^3$\\ %& $5714497=19\cdot 67^3$\\
$7837$ & $\not\equiv 0\pmod{3}$ & $37^3 1213$\\ %& $61442089=37^3 1213$\\
$12745$ & $\not\equiv 0\pmod{3}$ & $7^6 1381$\\ %& $162473269=7^6 1381$\\
$263135$ & $\not\equiv 0\pmod{3}$ & $7^3 43^3 2539$\\ %& $69240817639=7^3 43^3 2539$\\
$506370$ & $\equiv 12\pmod{27}$ & $3^3 193^3 1321$ %& $256412096019=3^3 193^3 1321$
\end{tabular}\\
\end{center}

\vfill\null
\columnbreak
\begin{center}
Table $2$: $-1\leq m\leq 10^7$ with class number $|\Cl(L_m)|=3$
\end{center}
\begin{center}
\begin{tabular}{rll}
$m$ & $m\pmod{3,9,27}$ & $m^2+3m+9$\\\hline
$6$ & $\equiv 6\pmod{9}$ & $3^2 7$\\ %& $63=3^2 7$\\
$9$ & $\equiv 0\pmod{9}$ & $3^2 13$\\ %& $117=3^2 13$\\
$13$ & $\not\equiv 0\pmod{3}$ & $7\cdot 31$\\ %& $217=7\cdot 31$\\
$14$ & $\not\equiv 0\pmod{3}$ & $13\cdot 19$\\ %& $247=13\cdot 19$\\
$15$ & $\equiv 6\pmod{9}$ & $3^2 31$\\ %& $279=3^2 31$\\
$18$ & $\equiv 0\pmod{9}$ & $3^2 43$\\ %& $387=3^2 43$\\
$19$ & $\not\equiv 0\pmod{3}$ & $7\cdot 61$\\ %& $427=7\cdot 61$\\
$20$ & $\not\equiv 0\pmod{3}$ & $7\cdot 67$\\ %& $469=7\cdot 67$\\
$21$ & $\equiv 21\pmod{27}$ & $3^3 19$\\ %& $513=3^3 19$\\
$22$ & $\not\equiv 0\pmod{3}$ & $13\cdot 43$\\ %& $559=13\cdot 43$\\
$30$ & $\equiv 3\pmod{27}$ & $3^3  37$\\ %& $999=3^3  37$\\
$41$ & $\not\equiv 0\pmod{3}$ & $7^2 37$\\ %& $1813=7^2 37$\\
$100$ & $\not\equiv 0\pmod{3}$ & $13^2 61$\\ %& $10309=13^2 61$\\
$147$ & $\equiv 12\pmod{27}$ & $3^3 19\cdot 43$\\ %& $22059=3^3 19\cdot 43$\\
$154$ & $\not\equiv 0\pmod{3}$ & $19^2 67$\\ %& $24187=19^2 67$\\
$174$ & $\equiv 12\pmod{27}$ & $3^3 7\cdot 163$\\ %& $30807=3^3 7\cdot 163$\\
$201$ & $\equiv 12\pmod{27}$ & $3^3 7^2 31$\\ %& $41013=3^3 7^2 31$\\
$271$ & $\not\equiv 0\pmod{3}$ & $7\cdot 103^2$\\ %& $74263=7\cdot 103^2$\\
$398$ & $\not\equiv 0\pmod{3}$ & $7\cdot 151^2$\\ %& $159607=7\cdot 151^2$\\
$629$ & $\not\equiv 0\pmod{3}$ & $7^3 19\cdot 61$\\ %& $397537=7^3 19\cdot 61$\\
$740$ & $\not\equiv 0\pmod{3}$ & $7^4 229$\\ %& $549829=7^4 229$\\
$876$ & $\equiv 12\pmod{27}$ & $3^3 19^2 79$\\ %& $770013=3^3 19^2 79$\\
$939$ & $\equiv 21\pmod{27}$ & $3^3 181^2$\\ %& $884547=3^3 181^2$\\
$1083$ & $\equiv 3\pmod{27}$ & $3^3 7^3 127$\\ %& $1176147=3^3 7^3 127$\\
$1497$ & $\equiv 12\pmod{27}$ & $3^3 7\cdot 109^2$\\ %& $2245509=3^3 7\cdot 109^2$\\
$3108$ & $\equiv 3\pmod{27}$ & $3^3 13^3 163$\\ %& $9668997=3^3 13^3 163$\\
$5258$ & $\not\equiv 0\pmod{3}$ & $19^3 37\cdot 109$\\ %& $27662347=19^3 37\cdot 109$\\
$7502$ & $\not\equiv 0\pmod{3}$ & $7^2 13^3 523$\\ %& $56302519=7^2 13^3 523$\\
$10927$ & $\not\equiv 0\pmod{3}$ & $19\cdot 31^3 211$\\ %& $119432119=19\cdot 31^3 211$\\
$222550$ & $\not\equiv 0\pmod{3}$ & $7^6 73^2 79$\\ %& $49529170159=7^6 73^2 79$\\
$1376233$ & $\not\equiv 0\pmod{3}$ & $7\cdot 13^3 43^3 1549$\\ %& $1894021398997=7\cdot 13^3 43^3 1549$
\end{tabular}\\
\end{center}~\\

\begin{center}
Table $3$: $-1\leq m\leq 10^7$ with class number $|\Cl(L_m)|=4$
$(\Cl(L_m)\simeq C_2\times C_2)$
\end{center}
\begin{center}
\begin{tabular}{rll}
$m$ & $m\pmod{3,9,27}$ & $m^2+3m+9$\\\hline
$11$ & $\not\equiv 0\pmod{3}$ & $163$\\
$17$ & $\not\equiv 0\pmod{3}$ & $349$\\
$23$ & $\not\equiv 0\pmod{3}$ & $607$\\
$25$ & $\not\equiv 0\pmod{3}$ & $709$\\
$29$ & $\not\equiv 0\pmod{3}$ & $937$\\
$120$ & $\equiv 12\pmod{27}$ & $3^3 547$\\ %& $14769=3^3 547$\\
$228$ & $\equiv 12\pmod{27}$ & $3^3 1951$\\ %& $52677=3^3 1951$\\
$4170$ & $\equiv 12\pmod{27}$ & $3^3 7^3 1879$\\ %& $17401419=3^3 7^3 1879$\\
$5088$ & $\equiv 12\pmod{27}$ & $3^3 7^3 2797$\\ %& $25903017=3^3 7^3 2797$\\
$101471$ & $\not\equiv 0\pmod{3}$ & $7^3 5479^2$\\ %& $10296668263=7^3 5479^2$\\
$6440111$& $\not\equiv 0\pmod{3}$ & $7^3 229^3 10069$ %&41475049012663, [7^3 229^3 10069
\end{tabular}\\
\end{center}

\vfill\null
\columnbreak
\begin{center}
Table $4$: $-1\leq m\leq 10^7$ with class number $|\Cl(L_m)|=7$
\end{center}
\begin{center}
\begin{tabular}{rll}
$m$ & $m\pmod{3,9,27}$ & $m^2+3m+9$\\\hline
$16$ & $\not\equiv 0\pmod{3}$ & $313$\\
$28$ & $\not\equiv 0\pmod{3}$ & $877$\\
$32$ & $\not\equiv 0\pmod{3}$ & $1129$\\
$38$ & $\not\equiv 0\pmod{3}$ & $1567$\\
$43$ & $\not\equiv 0\pmod{3}$ & $1987$\\
$49$ & $\not\equiv 0\pmod{3}$ & $2557$\\
$255$ & $\equiv 12\pmod{27}$ & $3^3 2437$\\ %& $65799=3^3 2437$\\
$309$ & $\equiv 12\pmod{27}$ & $3^3 3571$\\ %& $96417=3^3 3571$\\
$24614$ & $\not\equiv 0\pmod{3}$ & $43^3 7621$\\ %& $605922847=43^3 7621$\\
$612049$ & $\not\equiv 0\pmod{3}$ & $13^3 19^3 24859$ %& $374605814557=13^3 19^3 24859$
\end{tabular}\\
\end{center}~\\

\begin{center}
Table $5$: $-1\leq m\leq 10^7$ with class number $|\Cl(L_m)|=9$
$(\Cl(L_m)\simeq C_3\times C_3)$
\end{center}
\begin{center}
\begin{tabular}{rll}
$m$ & $m\pmod{3,9,27}$ & $m^2+3m+9$\\\hline
$24$ & $\equiv 6\pmod{9}$ & $3^2 73$\\ %& $657=3^2 73$\\
$27$ & $\equiv 0\pmod{9}$ & $3^2 7\cdot 13$\\ %& $819=3^2 7\cdot 13$\\
$33$ & $\equiv 6\pmod{9}$ & $3^2 7\cdot 19$\\ %& $1197=3^2 7\cdot 19$\\
$34$ & $\not\equiv 0\pmod{3}$ & $7\cdot 181$\\ %& $1267=7\cdot 181$\\
$35$ & $\not\equiv 0\pmod{3}$ & $13\cdot 103$\\ %& $1339=13\cdot 103$\\
$40$ & $\not\equiv 0\pmod{3}$ & $7\cdot 13\cdot 19$\\ %& $1729=7\cdot 13\cdot 19$\\
$47$ & $\not\equiv 0\pmod{3}$ & $7\cdot 337$\\ %& $2359=7\cdot 337$\\
$48$ & $\equiv 21\pmod{27}$ & $3^3  7\cdot 13$\\ %& $2457=3^3  7\cdot 13$\\
$52$ & $\not\equiv 0\pmod{3}$ & $19\cdot 151$\\ %& $2869=19\cdot 151$\\
$53$ & $\not\equiv 0\pmod{3}$ & $13\cdot 229$\\ %& $2977=13\cdot 229$\\
$75$ & $\equiv 21\pmod{27}$ & $3^3 7\cdot 31$\\ %& $5859=3^3 7\cdot 31$\\
$84$ & $\equiv 3\pmod{27}$ & $3^3 271$\\ %& $7317=3^3 271$\\
$90$ & $\equiv 0\pmod{9}$ & $3^2 7^2 19$\\ %& $8379=3^2 7^2 19$\\
$103$ & $\not\equiv 0\pmod{3}$ & $7^2 223$\\ %& $10927=7^2 223$\\
$139$ & $\not\equiv 0\pmod{3}$ & $7^2 13\cdot 31$\\ %& $19747=7^2 13\cdot 31$\\
$152$ & $\not\equiv 0\pmod{3}$ & $7^2 13\cdot 37$\\ %& $23569=7^2 13\cdot 37$\\
$204$ & $\equiv 6\pmod{9}$ & $3^2 13\cdot19^2$\\ %& $42237=3^2 13\cdot19^2$\\
$237$ & $\equiv 21\pmod{27}$ & $3^3 7^2 43$\\ %& $56889=3^3 7^2 43$\\
$374$ & $\not\equiv 0\pmod{3}$ & $37^2 103$\\ %& $141007=37^2 103$\\
$390$ & $\equiv 12\pmod{27}$ & $3^3 7\cdot 811$\\ %& $153279=3^3 7\cdot 811$\\
$417$ & $\equiv 12\pmod{27}$ & $3^3 13\cdot 499$\\ %& $175149=3^3 13\cdot 499$\\
$972$ & $\equiv 0\pmod{9}$ & $3^2 7^3 307$\\ %& $947709=3^2 7^3 307$\\
$1119$ & $\equiv 12\pmod{27}$ & $3^3 7^2 13\cdot 73$\\ %& $1255527=3^3 7^2 13\cdot 73$\\
$1315$ & $\not\equiv 0\pmod{3}$ & $7^3 31\cdot 163$\\ %& $1733179=7^3 31\cdot 163$\\
$1658$ & $\not\equiv 0\pmod{3}$ & $7^4 31\cdot 37$\\ %& $2753947=7^4 31\cdot 37$\\
$1769$ & $\not\equiv 0\pmod{3}$ & $7^3 13\cdot 19\cdot 37$\\ %& $3134677=7^3 13\cdot 19\cdot 37$\\
$3480$ & $\equiv 6\pmod{9}$ & $3^2 13^3 613$\\ %& $12120849=3^2 13^3 613$\\
$4059$ & $\equiv 0\pmod{9}$ & $3^2 7^5 109$\\ %& $16487667=3^2 7^5 109$\\
$6816$ & $\equiv 12\pmod{27}$ & $3^3  7^2 19\cdot 43^2$\\ %& $46478313=3^3  7^2 19\cdot 43^2$\\
$8457$ & $\equiv 6\pmod{9}$ & $3^2 19^4 61$\\ %& $71546229=3^2 19^4 61$\\
$12117$ & $\equiv 21\pmod{27}$ & $3^3 13\cdot19^3 61$\\ %& $146858049=3^3 13\cdot19^3 61$\\
$70509$ & $\equiv 12\pmod{27}$ & $3^3  7\cdot 31^3 883$\\ %& $4971730617=3^3  7\cdot 31^3 883$\\
$91858$ & $\not\equiv 0\pmod{3}$ & $73^3 109\cdot 199$\\ %& $8438167747=73^3 109\cdot 199$\\
$100952$ & $\not\equiv 0\pmod{3}$ & $7\cdot 79^3 2953$\\ %& $10191609169=7\cdot 79^3 2953$\\
$266748$ & $\equiv 6\pmod{9}$ & $3^2 7^2 13^3 271^2$\\ %& $71155295757=3^2 7^2 13^3 271^2$\\
$1360624$ & $\not\equiv 0\pmod{3}$ & $7^4 43\cdot 61^3 79$ %& $1851301751257=7^4 43\cdot 61^3 79$
\end{tabular}\\
\end{center}

\vfill\null
\columnbreak
\begin{center}
Table $6$: $-1\leq m\leq 10^7$ with class number $|\Cl(L_m)|=12$
$(\Cl(L_m)\simeq C_6\times C_2)$
\end{center}
\begin{center}
\begin{tabular}{rll}
$m$ & $m\pmod{3,9,27}$ & $m^2+3m+9$\\\hline
$26$ & $\not\equiv 0\pmod{3}$ & $7\cdot 109$\\ %& $763=7\cdot 109$\\
$36$ & $\equiv 0\pmod{9}$ & $3^2 157$\\ %& $1413=3^2 157$\\
$42$ & $\equiv 6\pmod{9}$ & $3^2 211$\\ %& $1899=3^2 211$\\
$44$ & $\not\equiv 0\pmod{3}$ & $31\cdot 67$\\ %& $2077=31\cdot 67$\\
$45$ & $\equiv 0\pmod{9}$ & $3^2 241$\\ %& $2169=3^2 241$\\
$55$ & $\not\equiv 0\pmod{3}$ & $7\cdot 457$\\ %& $3199=7\cdot 457$\\
$57$ & $\equiv 3\pmod{27}$ & $3^3 127$\\ %& $3429=3^3 127$\\
$59$ & $\not\equiv 0\pmod{3}$ & $19\cdot193$\\ %& $3667=19\cdot193$\\
$67$ & $\not\equiv 0\pmod{3}$ & $37\cdot 127$\\ %& $4699=37\cdot 127$\\
$188$ & $\not\equiv 0\pmod{3}$ & $7^2 733$\\ %& $35917=7^2 733$\\
$235$ & $\not\equiv 0\pmod{3}$ & $13^2 331$\\ %& $55939=13^2 331$\\
$269$ & $\not\equiv 0\pmod{3}$ & $13^2 433$\\ %& $73177=13^2 433$\\
$577$ & $\not\equiv 0\pmod{3}$ & $43^2 181$\\ %& $334669=43^2 181$\\
$716$ & $\not\equiv 0\pmod{3}$ & $13\cdot 199^2$\\ %& $514813=13\cdot 199^2$\\
$844$ & $\not\equiv 0\pmod{3}$ & $37\cdot 139^2$\\ %& $714877=37\cdot 139^2$\\
$1426$ & $\not\equiv 0\pmod{3}$ & $7^3 13\cdot 457$\\ %& $2037763=7^3 13\cdot 457$\\
$2344$ & $\not\equiv 0\pmod{3}$ & $7^343\cdot 373$\\ %& $5501377=7^343\cdot 373$\\
$2361$ & $\equiv 12\pmod{27}$ & $3^3 37^2 151$\\ %& $5581413=3^3 37^2 151$\\
$5305$ & $\not\equiv 0\pmod{3}$ & $7\cdot 13^3 1831$\\ %& $28158949=7\cdot 13^3 1831$\\
$5677$ & $\not\equiv 0\pmod{3}$ & $13^4 1129$\\ %& $32245369=13^4 1129$\\
$14349$ & $\equiv 12\pmod{27}$ & $3^3 7^3 37\cdot 601$ %& $205936857=3^3 7^3 37\cdot 601$
\end{tabular}\\
\end{center}~\\

\begin{center}
Table $7$: $-1\leq m\leq 10^7$ with class number $|\Cl(L_m)|=13$
\end{center}
\begin{center}
\begin{tabular}{rll}
$m$ & $m\pmod{3,9,27}$ & $m^2+3m+9$\\\hline
$31$ & $\not\equiv 0\pmod{3}$ & $1063$\\
$498$ & $\equiv 12\pmod{27}$ & $3^3 9241$\\ %& $249507=3^3 9241$\\
$36435$ & $\equiv 12\pmod{27}$ & $3^3 13^3 22381$ %& $1327618539=3^3 13^3 22381$
\end{tabular}\\
\end{center}
\end{multicols}
}

%%%%%%%%%%%%%%%%%%%%%%%%%%%%%%%%%%%%%%%%%%%%%%%%%%
\section{{Proof of Theorem \ref{thmain1} and Theorem \ref{thmain2}}}\label{S2}

{\it Proof of Theorem \ref{thmain1}.}\\
\indent 
We can use Theorem \ref{thLettl} (Lettl \cite{Let86})
or its improvement Theorem \ref{thLou} (Louboutin \cite{Lou02}). 
Here we just apply Theorem \ref{thLou}. 

It follows from Theorem \ref{thLou} that 
if $[\mathcal{O}_{L_m}:\bZ[\alpha]]=1$ and $m\geq 3423$, %$m\geq 5763$, 
then $\mathfrak{f}_m>2\sqrt{3}\cdot 10^4$ %$\mathfrak{f}_m=m^2+3m+9>10^5$ 
and hence 
%$h_m>\frac{(\mathfrak{f}_m)^}{(\log (m^2+3m+9))^2}>1000.105$.
$h_m>\frac{\mathfrak{f}_m}{e\,(\log (m^2+3m+9))^2\log \mathfrak{f}_m}>1000.329$. 
This implies that if $[\mathcal{O}_{L_m}:\bZ[\alpha]]=1$ and 
$h_m\leq 1000$, then $m\leq 3422$. %$m\leq 5762$. 

Similarly, 
if $[\mathcal{O}_{L_m}:\bZ[\alpha]]=3$ and $m\geq 6418$, %$m\geq 10743$, 
then $\mathfrak{f}_m=(m^2+3m+9)/3>2\sqrt{3}\cdot 10^4$ 
%$\mathfrak{f}_m=(m^2+3m+9)/3>10^5$ 
and 
$h_m>\frac{\mathfrak{f}_m}{e\,(\log (m^2+3m+9))^2\log \mathfrak{f}_m}$ 
$>$ $1000.069$. 
%$h_m>0.023\,\frac{(\mathfrak{f}_m)^{0.946}}{(\log (m^2+3m+9))^2}>1000.042$. 
This implies that if $[\mathcal{O}_{L_m}:\bZ[\alpha]]=3$ and 
$h_m\leq 1000$, then $m\leq 6417$. %$m\leq 10742$. 

If $[\mathcal{O}_{L_m}:\bZ[\alpha]]=27$ and $m\geq 22166$, %$m\geq 36764$, 
then $\mathfrak{f}_m=(m^2+3m+9)/27>2\sqrt{3}\cdot 10^4$ 
%$\mathfrak{f}_m=(m^2+3m+9)/27>10^5$ 
and 
$h_m>\frac{\mathfrak{f}_m}{e\,(\log (m^2+3m+9))^2\log \mathfrak{f}_m}>1000.010$. 
%$h_m>0.023\,\frac{(\mathfrak{f}_m)^{0.946}}{(\log (m^2+3m+9))^2}>1000.033$.
This implies that if $[\mathcal{O}_{L_m}:\bZ[\alpha]]=27$ and 
$h_m\leq 1000$, then $m\leq 22165$. %$m\leq 36763$. 

By using PARI/GP \cite{PARI2}, 
we can compute the class number $h_m$ of the simplest cubic field $L_m$ 
via the command ${\tt K=bnfinit(x^3-m*x^2-(m+3)*x-1,1)}$ and ${\tt K.no}$ 
under the assumption of the Generalized Riemann Hypothesis (GRH). 
In our case where 
$[\mathcal{O}_{L_m}:\bZ[\alpha]]=1$ (resp. $3$, $27$) 
and $m\leq 3422$ (resp. $m\leq 6417$, $m\leq 22165$), 
we can use the command {\tt bnfcertify(K)} to confirm that the obtained 
result is correct without the GRH. 
Indeed, we confirmed that {\tt bnfcertify(K)} returned $1$ in every case. 
See the PARI/GP computations as in Section \ref{S3}.\\

{\it Proof of Theorem \ref{thmain2}.}\\
\indent 
By using PARI/GP \cite{PARI2}, for $-1\leq m\leq 10^7$, 
we can compute the class number $h_m$ of the simplest cubic field $L_m$ 
and the ideal class group $\Cl(L_m)$ of $L_m$ 
via the command ${\tt K=bnfinit(x^3-m*x^2-(m+3)*x-1)}$ and 
${\tt K.clgp}$ under the assumption of the GRH. 
In the cases where $h_m<16$, 
we can use the command 
{\tt bnfcertify(K)} to confirm that 
the obtained result is correct without the GRH. 
Indeed, we confirmed that {\tt bnfcertify(K)} returned $1$ 
in each of these cases. 
See the PARI/GP computations as in Section \ref{S4}. 

%%%%%%%%%%%%%%%%%%%%%%%%%%%%%%%%%%%%%%%%%%%%%
\section{PARI/GP computations: The simplest cubic fields $L_m=\bQ(\alpha)$ with class number $h_m=|\Cl(L_m)|\leq 1000$ and $[\mathcal{O}_{L_m}:\bZ[\alpha]]=1,3,27$}\label{S3}

{\footnotesize 
\begin{verbatim}
gp > {
L=List();
for(m=-1,3422,
 if(m%3!=0&&isprime(m^2+3*m+9)==1,
 K=bnfinit(x^3-m*x^2-(m+3)*x-1,1);
 if(bnfcertify(K)==1,,print([m,"UnderGRH"]));
 cn=K.no;
 if(cn<=1000,listput(~L,[m,m^2+3*m+9,K.clgp[1..2]])
 )
 )
);
M=List();
for(i=1,1000,listput(~M,0);M[i]=select(x->x[3][1]==i,L);
if(M[i]!=List([]),print([i,M[i],length(M[i])]))
)
}
[1, List([[-1, 7, [1, []]], [1, 13, [1, []]], [2, 19, [1, []]], [4, 37, [1, []]], [7, 79, [1, []]], 
  [8, 97, [1, []]], [10, 139, [1, []]]]), 7]
[4, List([[11, 163, [4, [2, 2]]], [17, 349, [4, [2, 2]]], [23, 607, [4, [2, 2]]], [25, 709, [4, [2, 2]]], 
  [29, 937, [4, [2, 2]]]]), 5]
[7, List([[16, 313, [7, [7]]], [28, 877, [7, [7]]], [32, 1129, [7, [7]]], [38, 1567, [7, [7]]], 
 [43, 1987, [7, [7]]], [49, 2557, [7, [7]]]]), 6]
[13, List([[31, 1063, [13, [13]]]]), 1]
[16, List([[64, 4297, [16, [4, 4]]]]), 1]
[19, List([[37, 1489, [19, [19]]], [50, 2659, [19, [19]]], [56, 3313, [19, [19]]], [58, 3547, [19, [19]]], 
 [73, 5557, [19, [19]]], [88, 8017, [19, [19]]]]), 6]
[28, List([[85, 7489, [28, [14, 2]]], [95, 9319, [28, [14, 2]]]]), 2]
[31, List([[70, 5119, [31, [31]]], [94, 9127, [31, [31]]], [98, 9907, [31, [31]]]]), 3]
[37, List([[112, 12889, [37, [37]]], [140, 20029, [37, [37]]]]), 2]
[43, List([[107, 11779, [43, [43]]]]), 1]
[49, List([[91, 8563, [49, [49]]], [134, 18367, [49, [49]]]]), 2]
[52, List([[127, 16519, [52, [26, 2]]], [130, 17299, [52, [26, 2]]], [133, 18097, [52, [26, 2]]]]), 3]
[61, List([[122, 15259, [61, [61]]], [158, 25447, [61, [61]]], [164, 27397, [61, [61]]], 
 [172, 30109, [61, [61]]], [175, 31159, [61, [61]]]]), 5]
[64, List([[101, 10513, [64, [8, 8]]], [143, 20887, [64, [4, 4, 2, 2]]]]), 2]
[67, List([[155, 24499, [67, [67]]], [197, 39409, [67, [67]]]]), 2]
[73, List([[182, 33679, [73, [73]]], [214, 46447, [73, [73]]]]), 2]
[76, List([[163, 27067, [76, [38, 2]]], [169, 29077, [76, [38, 2]]], [179, 32587, [76, [38, 2]]]]), 3]
[91, List([[205, 42649, [91, [91]]], [238, 57367, [91, [91]]]]), 2]
[100, List([[136, 18913, [100, [10, 10]]], [220, 49069, [100, [10, 10]]]]), 2]
[112, List([[142, 20599, [112, [28, 4]]], [176, 31513, [112, [28, 4]]], [218, 48187, [112, [28, 4]]]]), 3]
[121, List([[239, 57847, [121, [11, 11]]]]), 1]
[124, List([[284, 81517, [124, [62, 2]]]]), 1]
[127, List([[121, 15013, [127, [127]]], [260, 68389, [127, [127]]]]), 2]
[133, List([[200, 40609, [133, [133]]], [302, 92119, [133, [133]]]]), 2]
[139, List([[322, 104659, [139, [139]]]]), 1]
[148, List([[277, 77569, [148, [74, 2]]]]), 1]
[169, List([[212, 45589, [169, [169]]], [224, 50857, [169, [169]]]]), 2]
[172, List([[262, 69439, [172, [86, 2]]]]), 1]
[175, List([[254, 65287, [175, [35, 5]]]]), 1]
[193, List([[403, 163627, [193, [193]]]]), 1]
[208, List([[290, 84979, [208, [52, 4]]], [305, 93949, [208, [52, 4]]]]), 2]
[217, List([[259, 67867, [217, [217]]], [392, 154849, [217, [217]]]]), 2]
[223, List([[206, 43063, [223, [223]]]]), 1]
[229, List([[304, 93337, [229, [229]]]]), 1]
[244, List([[332, 111229, [244, [122, 2]]]]), 1]
[247, List([[368, 136537, [247, [247]]]]), 1]
[259, List([[364, 133597, [259, [259]]]]), 1]
[268, List([[472, 224209, [268, [134, 2]]]]), 1]
[277, List([[367, 135799, [277, [277]]], [410, 169339, [277, [277]]]]), 2]
[292, List([[359, 129967, [292, [146, 2]]]]), 1]
[304, List([[406, 166063, [304, [76, 4]]]]), 1]
[313, List([[463, 215767, [313, [313]]]]), 1]
[316, List([[395, 157219, [316, [158, 2]]], [533, 285697, [316, [158, 2]]]]), 2]
[325, List([[343, 118687, [325, [65, 5]]], [388, 151717, [325, [65, 5]]]]), 2]
[343, List([[266, 71563, [343, [49, 7]]]]), 1]
[349, List([[281, 79813, [349, [349]]]]), 1]
[364, List([[301, 91513, [364, [182, 2]]], [310, 97039, [364, [182, 2]]]]), 2]
[388, List([[442, 196699, [388, [194, 2]]], [514, 265747, [388, [194, 2]]]]), 2]
[397, List([[487, 238639, [397, [397]]]]), 1]
[400, List([[371, 138763, [400, [20, 20]]], [473, 225157, [400, [10, 10, 2, 2]]]]), 2]
[403, List([[317, 101449, [403, [403]]], [539, 292147, [403, [403]]]]), 2]
[421, List([[346, 120763, [421, [421]]]]), 1]
[427, List([[574, 331207, [427, [427]]]]), 1]
[433, List([[658, 434947, [433, [433]]]]), 1]
[439, List([[518, 269887, [439, [439]]]]), 1]
[448, List([[494, 245527, [448, [56, 8]]], [617, 382549, [448, [28, 4, 2, 2]]], 
 [619, 385027, [448, [56, 8]]]]), 3]
[487, List([[634, 403867, [487, [487]]]]), 1]
[496, List([[428, 184477, [496, [124, 4]]]]), 1]
[499, List([[605, 367849, [499, [499]]]]), 1]
[511, List([[296, 88513, [511, [511]]]]), 1]
[523, List([[569, 325477, [523, [523]]], [595, 355819, [523, [523]]]]), 2]
[547, List([[380, 145549, [547, [547]]], [421, 178513, [547, [547]]]]), 2]
[553, List([[331, 110563, [553, [553]]], [361, 131413, [553, [553]]], [554, 308587, [553, [553]]], 
 [592, 352249, [553, [553]]], [598, 359407, [553, [553]]], [613, 377617, [553, [553]]]]), 6]
[559, List([[764, 585997, [559, [559]]]]), 1]
[571, List([[652, 427069, [571, [571]]]]), 1]
[592, List([[457, 210229, [592, [148, 4]]]]), 1]
[676, List([[704, 497737, [676, [338, 2]]]]), 1]
[691, List([[536, 288913, [691, [691]]]]), 1]
[703, List([[725, 527809, [703, [703]]], [767, 590599, [703, [703]]]]), 2]
[733, List([[562, 317539, [733, [733]]]]), 1]
[751, List([[497, 248509, [751, [751]]]]), 1]
[763, List([[520, 271969, [763, [763]]], [778, 607627, [763, [763]]]]), 2]
[772, List([[899, 810907, [772, [386, 2]]]]), 1]
[784, List([[644, 416677, [784, [196, 4]]]]), 1]
[811, List([[847, 719959, [811, [811]]]]), 1]
[832, List([[722, 523459, [832, [104, 8]]]]), 1]
[868, List([[773, 599857, [868, [434, 2]]]]), 1]
[871, List([[640, 411529, [871, [871]]]]), 1]
[883, List([[823, 679807, [883, [883]]]]), 1]
[892, List([[1033, 1070197, [892, [446, 2]]]]), 1]
[907, List([[910, 830839, [907, [907]]]]), 1]
[961, List([[436, 191413, [961, [31, 31]]], [476, 228013, [961, [961]]], [662, 440239, [961, [961]]]]), 3]
[991, List([[865, 750829, [991, [991]]]]), 1]
[997, List([[890, 794779, [997, [997]]]]), 1]
gp > length(L)
149
\end{verbatim}
}

%time = 30min, 4,687 ms. (Louboutin)
%time = 2h, 23min, 22,860 ms. (new)
%time = 3h, 1min, 30,031 ms. (old)

{\footnotesize
\begin{verbatim}
{
L=List();
for(m=-1,3422,
 if(m%3!=0&&isprime(m^2+3*m+9)==0&&issquarefree(m^2+3*m+9)==1,
 K=bnfinit(x^3-m*x^2-(m+3)*x-1,1);
 if(bnfcertify(K)==1,,print([m,"UnderGRH"]));
 cn=K.no;
 if(cn<=1000,listput(~L,[m,m^2+3*m+9,factor(m^2+3*m+9),K.clgp[1..2]])
 )
 )
);
print(length(L));
M=List();
for(i=1,1000,listput(~M,0);M[i]=select(x->x[4][1]==i,L);
if(M[i]!=List([]),print([i,M[i],length(M[i])]));
)
}
[3, List([[13, 217, [7, 1; 31, 1], [3, [3]]], [14, 247, [13, 1; 19, 1], [3, [3]]], [19, 427, [7, 1; 61, 1], [3, [3]]], 
 [20, 469, [7, 1; 67, 1], [3, [3]]], [22, 559, [13, 1; 43, 1], [3, [3]]]]), 5]
[9, List([[34, 1267, [7, 1; 181, 1], [9, [3, 3]]], [35, 1339, [13, 1; 103, 1], [9, [3, 3]]], 
 [40, 1729, [7, 1; 13, 1; 19, 1], [9, [3, 3]]], [47, 2359, [7, 1; 337, 1], [9, [3, 3]]], 
 [52, 2869, [19, 1; 151, 1], [9, [3, 3]]], [53, 2977, [13, 1; 229, 1], [9, [3, 3]]]]), 6]
[12, List([[26, 763, [7, 1; 109, 1], [12, [6, 2]]], [44, 2077, [31, 1; 67, 1], [12, [6, 2]]], 
 [55, 3199, [7, 1; 457, 1], [12, [6, 2]]], [59, 3667, [19, 1; 193, 1], [12, [6, 2]]], 
 [67, 4699, [37, 1; 127, 1], [12, [6, 2]]]]), 5]
[21, List([[46, 2263, [31, 1; 73, 1], [21, [21]]], [62, 4039, [7, 1; 577, 1], [21, [21]]], 
 [65, 4429, [43, 1; 103, 1], [21, [21]]], [68, 4837, [7, 1; 691, 1], [21, [21]]], 
 [74, 5707, [13, 1; 439, 1], [21, [21]]],  [77, 6169, [31, 1; 199, 1], [21, [21]]], 
 [80, 6649, [61, 1; 109, 1], [21, [21]]], [82, 6979, [7, 1; 997, 1], [21, [21]]]]), 8]
[27, List([[61, 3913, [7, 1; 13, 1; 43, 1], [27, [3, 3, 3]]], [97, 9709, [7, 1; 19, 1; 73, 1], [27, [3, 3, 3]]], 
 [104, 11137, [7, 1; 37, 1; 43, 1], [27, [3, 3, 3]]], [110, 12439, [7, 1; 1777, 1], [27, [9, 3]]]]), 4]
[36, List([[71, 5263, [19, 1; 277, 1], [36, [6, 6]]], [83, 7147, [7, 1; 1021, 1], [36, [6, 6]]], 
 [113, 13117, [13, 1; 1009, 1], [36, [6, 6]]], [118, 14287, [7, 1; 13, 1; 157, 1], [36, [6, 6]]], 
 [137, 19189, [31, 1; 619, 1], [36, [6, 6]]]]), 5]
[39, List([[89, 8197, [7, 1; 1171, 1], [39, [39]]], [92, 8749, [13, 1; 673, 1], [39, [39]]], 
 [109, 12217, [19, 1; 643, 1], [39, [39]]], [115, 13579, [37, 1; 367, 1], [39, [39]]], 
 [119, 14527, [73, 1; 199, 1], [39, [39]]], [124, 15757, [7, 1; 2251, 1], [39, [39]]]]), 6]
[48, List([[157, 25129, [13, 1; 1933, 1], [48, [6, 2, 2, 2]]]]), 1]
[57, List([[76, 6013, [7, 1; 859, 1], [57, [57]]], [125, 16009, [7, 1; 2287, 1], [57, [57]]], 
 [199, 40207, [31, 1; 1297, 1], [57, [57]]]]), 3]
[63, List([[79, 6487, [13, 1; 499, 1], [63, [21, 3]]], [86, 7663, [79, 1; 97, 1], [63, [21, 3]]], 
 [106, 11563, [31, 1; 373, 1], [63, [21, 3]]], [145, 21469, [7, 1; 3067, 1], [63, [21, 3]]], 
 [148, 22357, [79, 1; 283, 1], [63, [21, 3]]], [185, 34789, [19, 1; 1831, 1], [63, [21, 3]]]]), 6]
[81, List([[131, 17563, [7, 1; 13, 1; 193, 1], [81, [9, 3, 3]]], [173, 30457, [7, 1; 19, 1; 229, 1], [81, [9, 3, 3]]], 
 [178, 32227, [13, 1; 37, 1; 67, 1], [81, [9, 3, 3]]]]), 3]
[84, List([[151, 23263, [43, 1; 541, 1], [84, [42, 2]]], [227, 52219, [79, 1; 661, 1], [84, [42, 2]]], 
 [229, 53137, [7, 1; 7591, 1], [84, [42, 2]]], [242, 59299, [19, 1; 3121, 1], [84, [42, 2]]]]), 4]
[93, List([[116, 13813, [19, 1; 727, 1], [93, [93]]], [128, 16777, [19, 1; 883, 1], [93, [93]]], 
 [203, 41827, [151, 1; 277, 1], [93, [93]]]]), 3]
[108, List([[146, 21763, [7, 1; 3109, 1], [108, [18, 6]]], [170, 29419, [13, 1; 31, 1; 73, 1], [108, [6, 6, 3]]], 
 [194, 38227, [7, 1; 43, 1; 127, 1], [108, [6, 6, 3]]], [223, 50407, [7, 1; 19, 1; 379, 1], [108, [6, 6, 3]]], 
 [230, 53599, [7, 1; 13, 1; 19, 1; 31, 1], [108, [6, 6, 3]]], [248, 62257, [13, 1; 4789, 1], [108, [18, 6]]]]), 6]
[111, List([[160, 26089, [7, 1; 3727, 1], [111, [111]]], [161, 26413, [61, 1; 433, 1], [111, [111]]], 
 [167, 28399, [7, 1; 4057, 1], [111, [111]]], [190, 36679, [43, 1; 853, 1], [111, [111]]]]), 4]
[117, List([[202, 41419, [7, 1; 61, 1; 97, 1], [117, [39, 3]]], [208, 43897, [7, 1; 6271, 1], [117, [39, 3]]], 
 [209, 44317, [7, 1; 13, 1; 487, 1], [117, [39, 3]]], [232, 54529, [31, 1; 1759, 1], [117, [39, 3]]], 
 [253, 64777, [211, 1; 307, 1], [117, [39, 3]]]]), 5]
[129, List([[217, 47749, [13, 1; 3673, 1], [129, [129]]]]), 1]
[144, List([[166, 28063, [7, 1; 19, 1; 211, 1], [144, [12, 12]]], [263, 69967, [31, 1; 37, 1; 61, 1], [144, [12, 12]]], 
 [280, 79249, [19, 1; 43, 1; 97, 1], [144, [12, 12]]], [287, 83239, [13, 1; 19, 1; 337, 1], [144, [12, 12]]], 
 [329, 109237, [313, 1; 349, 1], [144, [12, 12]]]]), 5]
[147, List([[193, 37837, [157, 1; 241, 1], [147, [21, 7]]], [295, 87919, [13, 1; 6763, 1], [147, [21, 7]]], 
 [319, 102727, [43, 1; 2389, 1], [147, [147]]]]), 3]
[156, List([[149, 22657, [139, 1; 163, 1], [156, [78, 2]]], [184, 34417, [127, 1; 271, 1], [156, [78, 2]]], 
 [187, 35539, [7, 1; 5077, 1], [156, [78, 2]]], [274, 75907, [13, 1; 5839, 1], [156, [78, 2]]], 
 [307, 95179, [7, 1; 13597, 1], [156, [78, 2]]], [308, 95797, [13, 1; 7369, 1], [156, [78, 2]]]]), 6]
[171, List([[215, 46879, [7, 1; 37, 1; 181, 1], [171, [57, 3]]], [247, 61759, [151, 1; 409, 1], [171, [57, 3]]], 
 [344, 119377, [19, 1; 61, 1; 103, 1], [171, [57, 3]]]]), 3]
[183, List([[196, 39013, [13, 1; 3001, 1], [183, [183]]], [245, 60769, [67, 1; 907, 1], [183, [183]]], 
 [298, 89707, [109, 1; 823, 1], [183, [183]]], [340, 116629, [223, 1; 523, 1], [183, [183]]]]), 4]
[189, List([[257, 66829, [7, 1; 9547, 1], [189, [63, 3]]], [265, 71029, [7, 1; 73, 1; 139, 1], [189, [21, 3, 3]]], 
 [325, 106609, [19, 1; 31, 1; 181, 1], [189, [21, 3, 3]]], [328, 108577, [7, 1; 15511, 1], [189, [63, 3]]], 
 [353, 125677, [109, 1; 1153, 1], [189, [63, 3]]]]), 5]
[192, List([[283, 80947, [61, 1; 1327, 1], [192, [24, 8]]]]), 1]
[201, List([[278, 78127, [7, 1; 11161, 1], [201, [201]]], [320, 103369, [7, 1; 14767, 1], [201, [201]]]]), 2]
[219, List([[211, 45163, [19, 1; 2377, 1], [219, [219]]], [221, 49513, [67, 1; 739, 1], [219, [219]]], 
 [268, 72637, [19, 1; 3823, 1], [219, [219]]], [314, 99547, [7, 1; 14221, 1], [219, [219]]], 
 [379, 144787, [67, 1; 2161, 1], [219, [219]]], [452, 205669, [43, 1; 4783, 1], [219, [219]]]]), 6]
[225, List([[233, 54997, [43, 1; 1279, 1], [225, [15, 15]]], [244, 60277, [7, 1; 79, 1; 109, 1], [225, [15, 15]]], 
 [256, 66313, [13, 1; 5101, 1], [225, [15, 15]]], [293, 86737, [7, 1; 12391, 1], [225, [15, 15]]]]), 4]
[228, List([[191, 37063, [13, 1; 2851, 1], [228, [114, 2]]], [370, 138019, [7, 1; 19717, 1], [228, [114, 2]]]]), 2]
[237, List([[251, 63763, [7, 1; 9109, 1], [237, [237]]], [289, 84397, [37, 1; 2281, 1], [237, [237]]]]), 2]
[243, List([[338, 115267, [73, 1; 1579, 1], [243, [27, 9]]], 
 [385, 149389, [31, 1; 61, 1; 79, 1], [243, [9, 9, 3]]]]), 2]
[252, List([[181, 33313, [7, 1; 4759, 1], [252, [42, 6]]], [275, 76459, [157, 1; 487, 1], [252, [42, 6]]], 
 [292, 86149, [7, 1; 31, 1; 397, 1], [252, [42, 6]]], [313, 98917, [7, 1; 13, 1; 1087, 1], [252, [42, 6]]], 
 [350, 123559, [157, 1; 787, 1], [252, [42, 6]]], [358, 129247, [307, 1; 421, 1], [252, [42, 6]]], 
 [412, 170989, [7, 1; 13, 1; 1879, 1], [252, [42, 6]]]]), 7]
[273, List([[241, 58813, [103, 1; 571, 1], [273, [273]]], [272, 74809, [7, 1; 10687, 1], [273, [273]]], 
 [316, 100813, [73, 1; 1381, 1], [273, [273]]], [383, 147847, [7, 1; 21121, 1], [273, [273]]], 
 [437, 192289, [37, 1; 5197, 1], [273, [273]]]]), 5]
[279, List([[355, 127099, [7, 1; 67, 1; 271, 1], [279, [93, 3]]], [362, 132139, [7, 1; 43, 1; 439, 1], [279, [93, 3]]], 
 [413, 171817, [19, 1; 9043, 1], [279, [93, 3]]], [484, 235717, [73, 1; 3229, 1], [279, [93, 3]]]]), 4]
[291, List([[349, 122857, [7, 1; 17551, 1], [291, [291]]]]), 1]
[300, List([[236, 56413, [7, 1; 8059, 1], [300, [30, 10]]]]), 1]
[309, List([[409, 168517, [43, 1; 3919, 1], [309, [309]]]]), 1]
[324, List([[334, 112567, [7, 1; 13, 1; 1237, 1], [324, [18, 6, 3]]], [341, 117313, [7, 1; 16759, 1], [324, [18, 18]]], 
 [377, 143269, [7, 1; 97, 1; 211, 1], [324, [18, 6, 3]]], [382, 147079, [19, 1; 7741, 1], [324, [18, 18]]], 
 [503, 254527, [7, 1; 13, 1; 2797, 1], [324, [18, 6, 3]]]]), 5]
[327, List([[430, 186199, [13, 1; 14323, 1], [327, [327]]], [467, 219499, [7, 1; 31357, 1], [327, [327]]]]), 2]
[333, List([[448, 202057, [37, 1; 43, 1; 127, 1], [333, [111, 3]]]]), 1]
[336, List([[311, 97663, [127, 1; 769, 1], [336, [84, 4]]], [347, 121459, [13, 1; 9343, 1], [336, [42, 2, 2, 2]]], 
 [365, 134329, [13, 1; 10333, 1], [336, [84, 4]]], [400, 161209, [37, 1; 4357, 1], [336, [84, 4]]], 
 [407, 166879, [109, 1; 1531, 1], [336, [84, 4]]], [419, 176827, [7, 1; 25261, 1], [336, [42, 2, 2, 2]]], 
 [449, 202957, [31, 1; 6547, 1], [336, [84, 4]]], [454, 207487, [7, 1; 29641, 1], [336, [84, 4]]], 
 [458, 211147, [19, 1; 11113, 1], [336, [84, 4]]], [460, 212989, [7, 1; 30427, 1], [336, [84, 4]]]]), 10]
[351, List([[337, 114589, [19, 1; 37, 1; 163, 1], [351, [39, 3, 3]]], [373, 140257, [13, 1; 10789, 1], [351, [117, 3]]], 
 [425, 181909, [7, 1; 13, 1; 1999, 1], [351, [39, 3, 3]]], [445, 199369, [193, 1; 1033, 1], [351, [117, 3]]], 
 [455, 208399, [271, 1; 769, 1], [351, [117, 3]]]]), 5]
[372, List([[352, 124969, [13, 1; 9613, 1], [372, [186, 2]]], [389, 152497, [73, 1; 2089, 1], [372, [186, 2]]], 
 [427, 183619, [139, 1; 1321, 1], [372, [186, 2]]], [470, 222319, [19, 1; 11701, 1], [372, [186, 2]]], 
 [500, 251509, [163, 1; 1543, 1], [372, [186, 2]]], [505, 256549, [151, 1; 1699, 1], [372, [186, 2]]], 
 [547, 300859, [13, 1; 23143, 1], [372, [186, 2]]]]), 7]
[381, List([[394, 156427, [19, 1; 8233, 1], [381, [381]]]]), 1]
[387, List([[323, 105307, [31, 1; 43, 1; 79, 1], [387, [129, 3]]], [424, 181057, [331, 1; 547, 1], [387, [129, 3]]], 
 [550, 304159, [103, 1; 2953, 1], [387, [129, 3]]]]), 3]
[399, List([[485, 236689, [37, 1; 6397, 1], [399, [399]]], [529, 281437, [13, 1; 21649, 1], [399, [399]]], 
 [553, 307477, [19, 1; 16183, 1], [399, [399]]]]), 3]
[417, List([[226, 51763, [37, 1; 1399, 1], [417, [417]]], [416, 174313, [31, 1; 5623, 1], [417, [417]]], 
 [599, 360607, [13, 1; 27739, 1], [417, [417]]]]), 3]
[432, List([[434, 189667, [241, 1; 787, 1], [432, [36, 12]]], [439, 194047, [7, 1; 19, 1; 1459, 1], 
 [432, [6, 6, 6, 2]]], [440, 194929, [7, 1; 27847, 1], [432, [36, 12]]], 
 [524, 276157, [7, 1; 39451, 1], [432, [36, 12]]], [542, 295399, [13, 1; 31, 1; 733, 1], [432, [12, 12, 3]]], 
 [560, 315289, [13, 1; 79, 1; 307, 1], [432, [12, 12, 3]]]]), 6]
[441, List([[469, 221377, [13, 1; 17029, 1], [441, [147, 3]]], 
 [508, 259597, [13, 1; 19, 1; 1051, 1], [441, [21, 21]]], [509, 260617, [7, 1; 31, 1; 1201, 1], [441, [147, 3]]]]), 3]
[444, List([[557, 311929, [73, 1; 4273, 1], [444, [222, 2]]]]), 1]
[468, List([[418, 175987, [7, 1; 31, 1; 811, 1], [468, [78, 6]]], [475, 227059, [7, 1; 163, 1; 199, 1], [468, [78, 6]]], 
 [502, 253519, [7, 1; 36217, 1], [468, [78, 6]]], [517, 268849, [7, 1; 193, 1; 199, 1], [468, [78, 6]]], 
 [535, 287839, [73, 1; 3943, 1], [468, [78, 6]]], [559, 314167, [7, 1; 37, 1; 1213, 1], [468, [78, 6]]], 
 [575, 332359, [127, 1; 2617, 1], [468, [78, 6]]]]), 7]
[471, List([[386, 150163, [13, 1; 11551, 1], [471, [471]]]]), 1]
[507, List([[443, 197587, [13, 1; 15199, 1], [507, [507]]], [490, 241579, [13, 1; 18583, 1], [507, [507]]], 
 [647, 420559, [67, 1; 6277, 1], [507, [507]]]]), 3]
[513, List([[638, 408967, [13, 1; 163, 1; 193, 1], [513, [57, 3, 3]]]]), 1]
[516, List([[479, 230887, [373, 1; 619, 1], [516, [258, 2]]], [682, 467179, [139, 1; 3361, 1], [516, [258, 2]]]]), 2]
[525, List([[628, 396277, [7, 1; 56611, 1], [525, [105, 5]]]]), 1]
[549, List([[326, 107263, [13, 1; 37, 1; 223, 1], [549, [183, 3]]], 
 [481, 232813, [7, 1; 79, 1; 421, 1], [549, [183, 3]]], [488, 239617, [7, 1; 34231, 1], [549, [183, 3]]], 
 [493, 244537, [97, 1; 2521, 1], [549, [183, 3]]], [715, 513379, [349, 1; 1471, 1], [549, [183, 3]]]]), 5]
[567, List([[422, 179359, [67, 1; 2677, 1], [567, [63, 9]]], 
 [608, 371497, [7, 1; 73, 1; 727, 1], [567, [63, 3, 3]]], [685, 471289, [7, 1; 13, 1; 5179, 1], [567, [63, 3, 3]]], 
 [703, 496327, [13, 1; 73, 1; 523, 1], [567, [63, 3, 3]]]]), 4] 
[579, List([[610, 373939, [19, 1; 19681, 1], [579, [579]]], [707, 501979, [37, 1; 13567, 1], [579, [579]]], 
 [718, 517687, [79, 1; 6553, 1], [579, [579]]]]), 3]
[588, List([[512, 263689, [457, 1; 577, 1], [588, [294, 2]]], [523, 275107, [7, 1; 39301, 1], [588, [42, 14]]], 
 [532, 284629, [379, 1; 751, 1], [588, [42, 14]]], [548, 301957, [37, 1; 8161, 1], [588, [42, 14]]], 
 [712, 509089, [7, 1; 72727, 1], [588, [294, 2]]], [784, 617017, [97, 1; 6361, 1], [588, [294, 2]]]]), 6]
[597, List([[689, 476797, [397, 1; 1201, 1], [597, [597]]]]), 1]
[603, List([[623, 390007, [67, 1; 5821, 1], [603, [201, 3]]]]), 1]
[624, List([[614, 378847, [7, 1; 54121, 1], [624, [156, 4]]], 
 [637, 407689, [373, 1; 1093, 1], [624, [156, 4]]]]), 2]
[633, List([[530, 282499, [7, 1; 40357, 1], [633, [633]]], [679, 463087, [19, 1; 24373, 1], [633, [633]]]]), 2]
[651, List([[650, 424459, [7, 1; 60637, 1], [651, [651]]], [665, 444229, [307, 1; 1447, 1], [651, [651]]], 
 [713, 510517, [7, 1; 72931, 1], [651, [651]]]]), 3]
[657, List([[527, 279319, [19, 1; 61, 1; 241, 1], [657, [219, 3]]], [587, 346339, [7, 1; 49477, 1], [657, [219, 3]]], 
 [749, 563257, [43, 1; 13099, 1], [657, [219, 3]]], [790, 626479, [7, 1; 31, 1; 2887, 1], [657, [219, 3]]]]), 4]
[675, List([[464, 216697, [13, 1; 79, 1; 211, 1], [675, [15, 15, 3]]], 
 [568, 324337, [13, 1; 61, 1; 409, 1], [675, [15, 15, 3]]]]), 2]
[684, List([[511, 262663, [31, 1; 37, 1; 229, 1], [684, [114, 6]]], 
 [710, 506239, [43, 1; 61, 1; 193, 1], [684, [114, 6]]], 
 [734, 540967, [7, 1; 109, 1; 709, 1], [684, [114, 6]]]]), 3]
[687, List([[415, 173479, [283, 1; 613, 1], [687, [687]]]]), 1]
[711, List([[451, 204763, [13, 1; 19, 1; 829, 1], [711, [237, 3]]], 
 [572, 328909, [7, 1; 19, 1; 2473, 1], [711, [237, 3]]], 
 [620, 386269, [13, 1; 43, 1; 691, 1], [711, [237, 3]]], [667, 446899, [19, 1; 43, 1; 547, 1], [711, [237, 3]]], 
 [733, 539497, [7, 1; 37, 1; 2083, 1], [711, [237, 3]]], [760, 579889, [631, 1; 919, 1], [711, [237, 3]]], 
 [770, 595219, [37, 1; 16087, 1], [711, [237, 3]]]]), 7]
[729, List([[563, 318667, [223, 1; 1429, 1], [729, [27, 27]]], [604, 366637, [31, 1; 11827, 1], [729, [27, 27]]], 
 [659, 436267, [13, 1; 37, 1; 907, 1], [729, [27, 9, 3]]], 
 [664, 442897, [7, 1; 13, 1; 31, 1; 157, 1], [729, [9, 9, 3, 3]]], 
 [797, 637609, [7, 1; 79, 1; 1153, 1], [729, [27, 9, 3]]]]), 5]
[732, List([[461, 213913, [7, 1; 30559, 1], [732, [366, 2]]]]), 1]
[741, List([[376, 142513, [7, 1; 20359, 1], [741, [741]]], [431, 187063, [283, 1; 661, 1], [741, [741]]], 
 [643, 415387, [7, 1; 59341, 1], [741, [741]]]]), 3]
[756, List([[496, 247513, [7, 1; 19, 1; 1861, 1], [756, [42, 6, 3]]], 
 [538, 291067, [7, 1; 43, 1; 967, 1], [756, [42, 6, 3]]], 
 [584, 342817, [19, 1; 18043, 1], [756, [126, 6]]], [625, 392509, [13, 1; 109, 1; 277, 1], [756, [42, 6, 3]]], 
 [698, 489307, [7, 1; 13, 1; 19, 1; 283, 1], [756, [42, 6, 3]]], [808, 655297, [613, 1; 1069, 1], [756, [126, 6]]]]), 6]
[768, List([[668, 448237, [97, 1; 4621, 1], [768, [48, 16]]]]), 1]
[777, List([[526, 278263, [463, 1; 601, 1], [777, [777]]], [680, 464449, [109, 1; 4261, 1], [777, [777]]], 
 [799, 640807, [211, 1; 3037, 1], [777, [777]]]]), 3]
[804, List([[578, 335827, [499, 1; 673, 1], [804, [402, 2]]], [794, 632827, [13, 1; 48679, 1], [804, [402, 2]]], 
 [803, 647227, [7, 1; 92461, 1], [804, [402, 2]]]]), 3]
[819, List([[478, 229927, [31, 1; 7417, 1], [819, [273, 3]]], [590, 349879, [97, 1; 3607, 1], [819, [273, 3]]], 
 [670, 450919, [7, 1; 37, 1; 1741, 1], [819, [273, 3]]]]), 3]
[831, List([[802, 645619, [13, 1; 49663, 1], [831, [831]]]]), 1]
[837, List([[589, 348697, [157, 1; 2221, 1], [837, [279, 3]]], 
 [757, 575329, [31, 1; 67, 1; 277, 1], [837, [93, 3, 3]]], [769, 593677, [7, 1; 84811, 1], [837, [279, 3]]]]), 3]
[873, List([[635, 405139, [7, 1; 31, 1; 1867, 1], [873, [291, 3]]], 
 [656, 432313, [7, 1; 151, 1; 409, 1], [873, [291, 3]]], [692, 480949, [7, 1; 127, 1; 541, 1], [873, [291, 3]]], 
 [812, 661789, [19, 1; 61, 1; 571, 1], [873, [291, 3]]]]), 4]
[876, List([[641, 412813, [19, 1; 21727, 1], [876, [438, 2]]], [748, 561757, [7, 1; 80251, 1], [876, [438, 2]]], 
 [815, 666679, [13, 1; 51283, 1], [876, [438, 2]]]]), 3]
[900, List([[491, 242563, [43, 1; 5641, 1], [900, [30, 30]]], [892, 798349, [37, 1; 21577, 1], [900, [30, 30]]], 
 [914, 838147, [19, 1; 31, 1; 1423, 1], [900, [30, 30]]]]), 3]
[903, List([[674, 456307, [199, 1; 2293, 1], [903, [903]]], [854, 731887, [13, 1; 56299, 1], [903, [903]]]]), 2]
[912, List([[545, 298669, [7, 1; 42667, 1], [912, [228, 4]]], [848, 721657, [67, 1; 10771, 1], [912, [228, 4]]]]), 2]
[921, List([[775, 602959, [7, 1; 86137, 1], [921, [921]]]]), 1]
[927, List([[907, 825379, [19, 1; 43441, 1], [927, [309, 3]]]]), 1]
[939, List([[601, 363013, [7, 1; 51859, 1], [939, [939]]]]), 1]
[948, List([[683, 468547, [103, 1; 4549, 1], [948, [474, 2]]], [1043, 1090987, [853, 1; 1279, 1], [948, [474, 2]]]]), 2]
[972, List([[391, 154063, [7, 1; 13, 1; 1693, 1], [972, [18, 18, 3]]], 
 [571, 327763, [31, 1; 97, 1; 109, 1], [972, [18, 18, 3]]], [824, 681457, [7, 1; 67, 1; 1453, 1], [972, [18, 18, 3]]], 
 [845, 716569, [7, 1; 102367, 1], [972, [54, 18]]]]), 4]
[975, List([[967, 937999, [751, 1; 1249, 1], [975, [195, 5]]]]), 1]
[981, List([[763, 584467, [13, 1; 44959, 1], [981, [327, 3]]], [782, 613879, [7, 1; 87697, 1], [981, [327, 3]]]]), 2]
[993, List([[401, 162013, [19, 1; 8527, 1], [993, [993]]]]), 1]
[999, List([[743, 554287, [19, 1; 29173, 1], [999, [333, 3]]], [754, 570787, [7, 1; 73, 1; 1117, 1], [999, [111, 3, 3]]], 
 [859, 740467, [7, 1; 13, 1; 79, 1; 103, 1], [999, [111, 3, 3]]], 
 [893, 800137, [13, 1; 61, 1; 1009, 1], [999, [111, 3, 3]]]]), 4]
gp > length(L)
308
\end{verbatim}
}

%time = 1h, 5min, 11,937 ms.(Louboutin)
%time = 9h, 27min, 43,063 ms.(new)
%time = 8h, 48min, 46,516 ms.(old)

{\footnotesize 
\begin{verbatim}
{
L=List();
for(m=-1,3422,
 if((m%9==0||m%9==6)&&isprime((m^2+3*m+9)/9)==1,
 K=bnfinit(x^3-m*x^2-(m+3)*x-1,1);
 if(bnfcertify(K)==1,,print([m,"UnderGRH"]));
 cn=K.no;
 if(cn<=1000,listput(~L,[m,(m^2+3*m+9)/9,K.clgp[1..2]])
 )
 )
);
M=List();
for(i=1,1000,listput(~M,0);M[i]=select(x->x[3][1]==i,L);
if(M[i]!=List([]),print([i,M[i],length(M[i])]));
)
}
[3, List([[6, 7, [3, [3]]], [9, 13, [3, [3]]], [15, 31, [3, [3]]], [18, 43, [3, [3]]]]), 4]
[9, List([[24, 73, [9, [3, 3]]]]), 1]
[12, List([[36, 157, [12, [6, 2]]], [42, 211, [12, [6, 2]]], [45, 241, [12, [6, 2]]]]), 3]
[21, List([[60, 421, [21, [21]]], [63, 463, [21, [21]]]]), 2]
[39, List([[99, 1123, [39, [39]]]]), 1]
[48, List([[72, 601, [48, [12, 4]]]]), 1]
[57, List([[123, 1723, [57, [57]]]]), 1]
[63, List([[51, 307, [63, [21, 3]]], [81, 757, [63, [21, 3]]], [162, 2971, [63, [21, 3]]]]), 3]
[84, List([[150, 2551, [84, [42, 2]]]]), 1]
[111, List([[114, 1483, [111, [111]]]]), 1]
[147, List([[234, 6163, [147, [147]]]]), 1]
[156, List([[177, 3541, [156, [78, 2]]], [225, 5701, [156, [78, 2]]]]), 2]
[189, List([[186, 3907, [189, [63, 3]]], [240, 6481, [189, [63, 3]]], 
 [267, 8011, [189, [63, 3]]], [297, 9901, [189, [63, 3]]]]), 4]
[201, List([[207, 4831, [201, [201]]]]), 1]
[225, List([[270, 8191, [225, [15, 15]]]]), 1]
[228, List([[315, 11131, [228, [114, 2]]]]), 1]
[252, List([[213, 5113, [252, [42, 6]]]]), 1]
[273, List([[198, 4423, [273, [273]]], [357, 14281, [273, [273]]]]), 2]
[300, List([[171, 3307, [300, [30, 10]]]]), 1]
[309, List([[330, 12211, [309, [309]]], [333, 12433, [309, [309]]]]), 2]
[336, List([[231, 6007, [336, [84, 4]]], [393, 17293, [336, [84, 4]]]]), 2]
[372, List([[423, 20023, [372, [186, 2]]], [465, 24181, [372, [186, 2]]]]), 2]
[399, List([[414, 19183, [399, [399]]]]), 1]
[471, List([[528, 31153, [471, [471]]]]), 1]
[525, List([[582, 37831, [525, [105, 5]]]]), 1]
[567, List([[351, 13807, [567, [63, 9]]], [429, 20593, [567, [63, 9]]]]), 2]
[603, List([[459, 23563, [603, [201, 3]]], [645, 46441, [603, [201, 3]]]]), 2]
[624, List([[303, 10303, [624, [156, 4]]]]), 1]
[633, List([[504, 28393, [633, [633]]]]), 1]
[657, List([[483, 26083, [657, [219, 3]]]]), 1]
[687, List([[450, 22651, [687, [687]]]]), 1]
[723, List([[519, 30103, [723, [723]]]]), 1]
[732, List([[693, 53593, [732, [366, 2]]]]), 1]
[777, List([[627, 43891, [777, [777]]]]), 1]
[819, List([[567, 35911, [819, [273, 3]]]]), 1]
[876, List([[735, 60271, [876, [438, 2]]]]), 1]
[975, List([[441, 21757, [975, [195, 5]]]]), 1]
[993, List([[654, 47743, [993, [993]]]]), 1]
gp > length(L)
56
\end{verbatim}
}

%time = 11min, 31,782 ms.(Louboutin)
%time = 49min, 37,876 ms.(new2)
%time = 1h, 55min, 27,813 ms.(new)
%time = 1h, 18min, 49,593 ms.(old)

{\footnotesize 
\begin{verbatim}
{
L=List();
for(m=-1,3422,
 if(m!=0&&(m%9==0||m%9==6)&&isprime((m^2+3*m+9)/9)==0&&issquarefree((m^2+3*m+9)/9)==1,
 K=bnfinit(x^3-m*x^2-(m+3)*x-1,1);
 if(bnfcertify(K)==1,,print([m,"UnderGRH"]));
 cn=K.no;
 if(cn<=1000,listput(~L,[m,(m^2+3*m+9)/9,factor((m^2+3*m+9)/9),K.clgp[1..2]])
 )
 )
);
M=List();
for(i=1,1000,listput(~M,0);M[i]=select(x->x[4][1]==i,L);
if(M[i]!=List([]),print([i,M[i],length(M[i])]));
)
}
[9, List([[27, 91, [7, 1; 13, 1], [9, [3, 3]]], [33, 133, [7, 1; 19, 1], [9, [3, 3]]]]), 2]
[27, List([[69, 553, [7, 1; 79, 1], [27, [3, 3, 3]]], [78, 703, [19, 1; 37, 1], [27, [3, 3, 3]]]]), 2]
[36, List([[87, 871, [13, 1; 67, 1], [36, [6, 6]]]]), 1]
[63, List([[105, 1261, [13, 1; 97, 1], [63, [21, 3]]], [108, 1333, [31, 1; 43, 1], [63, [21, 3]]], 
 [117, 1561, [7, 1; 223, 1], [63, [21, 3]]], [132, 1981, [7, 1; 283, 1], [63, [21, 3]]]]), 4]
[81, List([[144, 2353, [13, 1; 181, 1], [81, [9, 3, 3]]], [153, 2653, [7, 1; 379, 1], [81, [9, 3, 3]]]]), 2]
[108, List([[126, 1807, [13, 1; 139, 1], [108, [6, 6, 3]]], [135, 2071, [19, 1; 109, 1], [108, [6, 6, 3]]], 
 [168, 3193, [31, 1; 103, 1], [108, [6, 6, 3]]], [189, 4033, [37, 1; 109, 1], [108, [6, 6, 3]]]]), 4]
[117, List([[96, 1057, [7, 1; 151, 1], [117, [39, 3]]], [159, 2863, [7, 1; 409, 1], [117, [39, 3]]], 
 [195, 4291, [7, 1; 613, 1], [117, [39, 3]]]]), 3]
[144, List([[180, 3661, [7, 1; 523, 1], [144, [6, 6, 2, 2]]]]), 1]
[171, List([[249, 6973, [19, 1; 367, 1], [171, [57, 3]]], [294, 9703, [31, 1; 313, 1], [171, [57, 3]]]]), 2]
[189, List([[141, 2257, [37, 1; 61, 1], [189, [21, 3, 3]]], [243, 6643, [7, 1; 13, 1; 73, 1], [189, [21, 3, 3]]]]), 2]
[243, List([[222, 5551, [7, 1; 13, 1; 61, 1], [243, [9, 3, 3, 3]]], [258, 7483, [7, 1; 1069, 1], [243, [9, 9, 3]]], 
 [288, 9313, [67, 1; 139, 1], [243, [9, 9, 3]]]]), 3]
[252, List([[252, 7141, [37, 1; 193, 1], [252, [42, 6]]], [339, 12883, [13, 1; 991, 1], [252, [42, 6]]]]), 2]
[324, List([[360, 14521, [13, 1; 1117, 1], [324, [18, 6, 3]]]]), 1]
[333, List([[216, 5257, [7, 1; 751, 1], [333, [111, 3]]], [279, 8743, [7, 1; 1249, 1], [333, [111, 3]]], 
 [285, 9121, [7, 1; 1303, 1], [333, [111, 3]]], [312, 10921, [67, 1; 163, 1], [333, [111, 3]]], 
 [342, 13111, [7, 1; 1873, 1], [333, [111, 3]]]]), 5]
[351, List([[276, 8557, [43, 1; 199, 1], [351, [39, 3, 3]]]]), 1]
[432, List([[375, 15751, [19, 1; 829, 1], [432, [12, 12, 3]]], 
 [432, 20881, [7, 1; 19, 1; 157, 1], [432, [12, 12, 3]]]]), 2]
[441, List([[378, 16003, [13, 1; 1231, 1], [441, [147, 3]]]]), 1]
[468, List([[522, 30451, [37, 1; 823, 1], [468, [78, 6]]]]), 1]
[513, List([[474, 25123, [7, 1; 37, 1; 97, 1], [513, [57, 3, 3]]]]), 1]
[549, List([[537, 32221, [7, 1; 4603, 1], [549, [183, 3]]]]), 1]
[567, List([[261, 7657, [13, 1; 19, 1; 31, 1], [567, [21, 3, 3, 3]]], [324, 11773, [61, 1; 193, 1], [567, [63, 3, 3]]], 
 [405, 18361, [7, 1; 43, 1; 61, 1], [567, [21, 3, 3, 3]]], [549, 33673, [151, 1; 223, 1], [567, [63, 3, 3]]]]), 4]
[576, List([[540, 32581, [31, 1; 1051, 1], [576, [12, 12, 2, 2]]]]), 1]
[603, List([[420, 19741, [19, 1; 1039, 1], [603, [201, 3]]]]), 1]
[657, List([[366, 15007, [43, 1; 349, 1], [657, [219, 3]]], [387, 16771, [31, 1; 541, 1], [657, [219, 3]]], 
 [603, 40603, [19, 1; 2137, 1], [657, [219, 3]]]]), 3]
[684, List([[369, 15253, [7, 1; 2179, 1], [684, [114, 6]]], [402, 18091, [79, 1; 229, 1], [684, [114, 6]]], 
 [468, 24493, [7, 1; 3499, 1], [684, [114, 6]]], [672, 50401, [13, 1; 3877, 1], [684, [114, 6]]]]), 4]
[711, List([[396, 17557, [97, 1; 181, 1], [711, [237, 3]]]]), 1]
[729, List([[447, 22351, [7, 1; 31, 1; 103, 1], [729, [9, 9, 3, 3]]], 
 [477, 25441, [13, 1; 19, 1; 103, 1], [729, [9, 9, 3, 3]]], 
 [690, 53131, [13, 1; 61, 1; 67, 1], [729, [9, 9, 3, 3]]]]), 3]
[756, List([[306, 10507, [7, 1; 19, 1; 79, 1], [756, [42, 6, 3]]]]), 1]
[837, List([[321, 11557, [7, 1; 13, 1; 127, 1], [837, [93, 3, 3]]], 
 [594, 39403, [7, 1; 13, 1; 433, 1], [837, [93, 3, 3]]], [612, 41821, [13, 1; 3217, 1], [837, [93, 3, 3]]], 
 [657, 48181, [7, 1; 6883, 1], [837, [93, 3, 3]]]]), 4]
[900, List([[780, 67861, [79, 1; 859, 1], [900, [30, 30]]]]), 1]
[972, List([[546, 33307, [19, 1; 1753, 1], [972, [18, 18, 3]]], [558, 34783, [7, 1; 4969, 1], [972, [18, 18, 3]]], 
 [648, 46873, [19, 1; 2467, 1], [972, [18, 18, 3]]]]), 3]
gp > length(L)
67
\end{verbatim}
}

%time = 28min, 44,547 ms.(Louboutin)
%time = 4h, 29min, 28,063 ms.(new)
%time = 3h, 10min, 57,031 ms.(old)

%Type C

%%%%%%%%%%%%%%%%%%%%%%%%%%%%%%%%%%%%%%%%%%%%%%%%%%%%%
%
{\footnotesize 
\begin{verbatim}
{
L=List();
for(m=-1,6417,
 if((m%27==3||m%27==21)&&isprime((m^2+3*m+9)/27)==1,
 K=bnfinit(x^3-m*x^2-(m+3)*x-1,1);
 if(bnfcertify(K)==1,,print([m,"UnderGRH"]));
 cn=K.no;
 if(cn<=1000,listput(~L,[m,(m^2+3*m+9)/27,K.clgp[1..2]])
 )
 )
);
M=List();
for(i=1,1000,listput(~M,0);M[i]=select(x->x[3][1]==i,L);
if(M[i]!=List([]),print([i,M[i],length(M[i])]));
)
}
[3, List([[21, 19, [3, [3]]], [30, 37, [3, [3]]]]), 2]
[9, List([[84, 271, [9, [3, 3]]]]), 1]
[12, List([[57, 127, [12, [6, 2]]]]), 1]
[21, List([[102, 397, [21, [21]]], [129, 631, [21, [21]]]]), 2]
[39, List([[210, 1657, [39, [39]]]]), 1]
[48, List([[273, 2791, [48, [12, 4]]]]), 1]
[84, List([[219, 1801, [84, [42, 2]]]]), 1]
[111, List([[372, 5167, [111, [111]]]]), 1]
[117, List([[156, 919, [117, [39, 3]]]]), 1]
[147, List([[291, 3169, [147, [147]]]]), 1]
[156, List([[435, 7057, [156, [78, 2]]]]), 1]
[183, List([[345, 4447, [183, [183]]]]), 1]
[189, List([[246, 2269, [189, [63, 3]]]]), 1]
[228, List([[381, 5419, [228, [114, 2]]], [669, 16651, [228, [114, 2]]]]), 2]
[243, List([[570, 12097, [243, [27, 9]]]]), 1]
[252, List([[408, 6211, [252, [42, 6]]]]), 1]
[324, List([[732, 19927, [324, [18, 18]]]]), 1]
[333, List([[813, 24571, [333, [111, 3]]]]), 1]
[372, List([[777, 22447, [372, [186, 2]]]]), 1]
[417, List([[597, 13267, [417, [417]]]]), 1]
[432, List([[561, 11719, [432, [36, 12]]]]), 1]
[525, List([[858, 27361, [525, [105, 5]]]]), 1]
[567, List([[723, 19441, [567, [63, 9]]]]), 1]
[588, List([[840, 26227, [588, [294, 2]]]]), 1]
[687, List([[1155, 49537, [687, [687]]]]), 1]
[732, List([[948, 33391, [732, [366, 2]]], [1263, 59221, [732, [366, 2]]]]), 2]
[756, List([[1128, 47251, [756, [126, 6]]]]), 1]
[777, List([[1074, 42841, [777, [777]]]]), 1]
[900, List([[975, 35317, [900, [30, 30]]]]), 1]
[921, List([[1407, 73477, [921, [921]]]]), 1]
gp > length(L)
34
\end{verbatim}
}

%time = 8min, 10,031 ms.(Louboutin)
%time = 4h, 29min, 28,063 ms.(new)
%time = 3h, 10min, 57,031 ms.(old)

{\footnotesize 
\begin{verbatim}
{
L=List();
for(m=-1,6417,
 if(m!=3&&(m%27==3||m%27==21)&&isprime((m^2+3*m+9)/27)==0&&issquarefree((m^2+3*m+9)/27)==1,
 K=bnfinit(x^3-m*x^2-(m+3)*x-1,1);
 if(bnfcertify(K)==1,,print([m,"UnderGRH"]));
 cn=K.no;
 if(cn<=1000,listput(~L,[m,(m^2+3*m+9)/27,factor((m^2+3*m+9)/27),K.clgp[1..2]])
 )
 )
);
print(length(L));
M=List();
for(i=1,1000,listput(~M,0);M[i]=select(x->x[4][1]==i,L);
if(M[i]!=List([]),print([i,M[i],length(M[i])]));
)
}
[9, List([[48, 91, [7, 1; 13, 1], [9, [3, 3]]], [75, 217, [7, 1; 31, 1], [9, [3, 3]]]]), 2]
[27, List([[183, 1261, [13, 1; 97, 1], [27, [3, 3, 3]]]]), 1]
[36, List([[111, 469, [7, 1; 67, 1], [36, [6, 6]]], [138, 721, [7, 1; 103, 1], [36, [6, 6]]], 
 [165, 1027, [13, 1; 79, 1], [36, [6, 6]]], [192, 1387, [19, 1; 73, 1], [36, [6, 6]]]]), 4]
[63, List([[264, 2611, [7, 1; 373, 1], [63, [21, 3]]]]), 1]
[81, List([[318, 3781, [19, 1; 199, 1], [81, [9, 3, 3]]]]), 1]
[108, List([[300, 3367, [7, 1; 13, 1; 37, 1], [108, [6, 6, 3]]]]), 1]
[117, List([[327, 3997, [7, 1; 571, 1], [117, [39, 3]]]]), 1]
[144, List([[399, 5941, [13, 1; 457, 1], [144, [12, 12]]], [462, 7957, [73, 1; 109, 1], [144, [12, 12]]]]), 2]
[171, List([[354, 4681, [31, 1; 151, 1], [171, [57, 3]]], [480, 8587, [31, 1; 277, 1], [171, [57, 3]]]]), 2]
[189, List([[489, 8911, [7, 1; 19, 1; 67, 1], [189, [21, 3, 3]]], 
 [588, 12871, [61, 1; 211, 1], [189, [21, 3, 3]]]]), 2]
[225, List([[507, 9577, [61, 1; 157, 1], [225, [15, 15]]]]), 1]
[252, List([[543, 10981, [79, 1; 139, 1], [252, [42, 6]]]]), 1]
[279, List([[426, 6769, [7, 1; 967, 1], [279, [93, 3]]]]), 1]
[324, List([[615, 14077, [7, 1; 2011, 1], [324, [18, 6, 3]]], 
 [705, 18487, [7, 1; 19, 1; 139, 1], [324, [6, 6, 3, 3]]]]), 2]
[333, List([[453, 7651, [7, 1; 1093, 1], [333, [111, 3]]], [624, 14491, [43, 1; 337, 1], [333, [111, 3]]]]), 2]
[468, List([[750, 20917, [13, 1; 1609, 1], [468, [78, 6]]], [894, 29701, [7, 1; 4243, 1], [468, [78, 6]]]]), 2]
[513, List([[516, 9919, [7, 1; 13, 1; 109, 1], [513, [57, 3, 3]]], 
 [1002, 37297, [13, 1; 19, 1; 151, 1], [513, [57, 3, 3]]]]), 2]
[549, List([[804, 24031, [7, 1; 3433, 1], [549, [183, 3]]]]), 1]
[567, List([[534, 10621, [13, 1; 19, 1; 43, 1], [567, [21, 3, 3, 3]]], [696, 18019, [37, 1; 487, 1], [567, [63, 3, 3]]], 
 [867, 27937, [7, 1; 13, 1; 307, 1], [567, [21, 3, 3, 3]]]]), 3]
[576, List([[759, 21421, [31, 1; 691, 1], [576, [12, 12, 2, 2]]], 
 [885, 29107, [13, 1; 2239, 1], [576, [12, 12, 2, 2]]]]), 2]
[657, List([[1209, 54271, [7, 1; 7753, 1], [657, [219, 3]]]]), 1]
[711, List([[651, 15769, [13, 1; 1213, 1], [711, [237, 3]]]]), 1]
[729, List([[1029, 39331, [37, 1; 1063, 1], [729, [27, 9, 3]]], 
 [1398, 72541, [7, 1; 43, 1; 241, 1], [729, [9, 9, 3, 3]]]]), 2]
[756, List([[1047, 40717, [19, 1; 2143, 1], [756, [42, 6, 3]]], [1137, 48007, [61, 1; 787, 1], [756, [42, 6, 3]]]]), 2]
[837, List([[1020, 38647, [7, 1; 5521, 1], [837, [93, 3, 3]]]]), 1]
[999, List([[912, 30907, [31, 1; 997, 1], [999, [111, 3, 3]]], [1182, 51877, [7, 1; 7411, 1], [999, [111, 3, 3]]], 
 [1218, 55081, [13, 1; 19, 1; 223, 1], [999, [111, 3, 3]]], [1515, 85177, [19, 1; 4483, 1], [999, [111, 3, 3]]]]), 4]
gp > length(L)
45
\end{verbatim}
}

% time = 20min, 40,328 ms.(Louboutin)
%time = 3h, 33min, 828 ms.(new)
%time = 2h, 26min, 20,688 ms.(old)

%Type B

{\footnotesize 
\begin{verbatim}
{
L=List();
for(m=-1,22165,
 if(m%27==12&&isprime((m^2+3*m+9)/27)==1,
 K=bnfinit(x^3-m*x^2-(m+3)*x-1,1);
 if(bnfcertify(K)==1,,print([m,"UnderGRH"]));
 cn=K.no;
 if(cn<=1000,listput(~L,[m,(m^2+3*m+9)/27,K.clgp[1..2]])
 )
 )
);
M=List();
for(i=1,1000,listput(~M,0);M[i]=select(x->x[3][1]==i,L);
if(M[i]!=List([]),print([i,M[i],length(M[i])]));
)
}
[1, List([[12, 7, [1, []]], [39, 61, [1, []]], [93, 331, [1, []]]]), 3]
[4, List([[120, 547, [4, [2, 2]]], [228, 1951, [4, [2, 2]]]]), 2]
[7, List([[255, 2437, [7, [7]]], [309, 3571, [7, [7]]]]), 2]
[13, List([[498, 9241, [13, [13]]]]), 1]
[28, List([[336, 4219, [28, [14, 2]]], [822, 25117, [28, [14, 2]]]]), 2]
[31, List([[795, 23497, [31, [31]]]]), 1]
[37, List([[471, 8269, [37, [37]]]]), 1]
[49, List([[444, 7351, [49, [49]]], [525, 10267, [49, [7, 7]]]]), 2]
[97, List([[1848, 126691, [97, [97]]]]), 1]
[109, List([[606, 13669, [109, [109]]]]), 1]
[133, List([[1227, 55897, [133, [133]]]]), 1]
[148, List([[1767, 115837, [148, [74, 2]]]]), 1]
[175, List([[2577, 246247, [175, [35, 5]]]]), 1]
[193, List([[1281, 60919, [193, [193]]]]), 1]
[196, List([[2685, 267307, [196, [98, 2]]]]), 1]
[211, List([[1551, 89269, [211, [211]]]]), 1]
[247, List([[2604, 251431, [247, [247]]]]), 1]
[307, List([[1659, 102121, [307, [307]]]]), 1]
[316, List([[2334, 202021, [316, [158, 2]]], [2982, 329677, [316, [158, 2]]]]), 2]
[331, List([[3063, 347821, [331, [331]]]]), 1]
[343, List([[2496, 231019, [343, [343]]]]), 1]
[532, List([[1578, 92401, [532, [266, 2]]], [5169, 990151, [532, [266, 2]]]]), 2]
[541, List([[2766, 283669, [541, [541]]]]), 1]
[553, List([[3171, 372769, [553, [553]]]]), 1]
[604, List([[1740, 112327, [604, [302, 2]]]]), 1]
[628, List([[2145, 170647, [628, [314, 2]]]]), 1]
[637, List([[5412, 1085407, [637, [637]]]]), 1]
[652, List([[1983, 145861, [652, [326, 2]]]]), 1]
[661, List([[5034, 939121, [661, [661]]]]), 1]
[688, List([[1416, 74419, [688, [172, 4]]], [4602, 784897, [688, [172, 4]]], 
 [4872, 879667, [688, [86, 2, 2, 2]]]]), 3]
[739, List([[4035, 603457, [739, [739]]]]), 1]
[769, List([[3495, 452797, [769, [769]]], [4683, 812761, [769, [769]]]]), 2]
[787, List([[4953, 909151, [787, [787]]]]), 1]
[889, List([[4062, 611557, [889, [889]]]]), 1]
gp > length(L)
45
\end{verbatim}
}

%time = 14min, 38,797 ms.(Louboutin)
%time = 1h, 9min, 22,501 ms.(new2)
%time = 2h, 30min, 26,609 ms.(new)
%time = 1h, 43min, 9,250 ms.(old)

{\footnotesize 
\begin{verbatim}
{
L=List();
for(m=-1,22165,
 if(m%27==12&&isprime((m^2+3*m+9)/27)==0&&issquarefree((m^2+3*m+9)/27)==1,
 K=bnfinit(x^3-m*x^2-(m+3)*x-1,1);
 if(bnfcertify(K)==1,,print([m,"UnderGRH"]));
 cn=K.no;
 if(cn<=1000,listput(~L,[m,(m^2+3*m+9)/27,factor((m^2+3*m+9)/27),K.clgp[1..2]])
 )
 )
);
M=List();
for(i=1,1000,listput(~M,0);M[i]=select(x->x[4][1]==i,L);
if(M[i]!=List([]),print([i,M[i],length(M[i])]));
)
}
[3, List([[147, 817, [19, 1; 43, 1], [3, [3]]], [174, 1141, [7, 1; 163, 1], [3, [3]]]]), 2]
[9, List([[390, 5677, [7, 1; 811, 1], [9, [3, 3]]], [417, 6487, [13, 1; 499, 1], [9, [3, 3]]]]), 2]
[21, List([[552, 11347, [7, 1; 1621, 1], [21, [21]]], [579, 12481, [7, 1; 1783, 1], [21, [21]]]]), 2]
[27, List([[363, 4921, [7, 1; 19, 1; 37, 1], [27, [3, 3, 3]]], 
 [633, 14911, [13, 1; 31, 1; 37, 1], [27, [3, 3, 3]]]]), 2]
[36, List([[282, 2977, [13, 1; 229, 1], [36, [6, 6]]], [714, 18961, [67, 1; 283, 1], [36, [6, 6]]]]), 2]
[39, List([[903, 30301, [157, 1; 193, 1], [39, [39]]], [984, 35971, [13, 1; 2767, 1], [39, [39]]]]), 2]
[48, List([[660, 16207, [19, 1; 853, 1], [48, [12, 4]]]]), 1]
[57, List([[957, 34027, [7, 1; 4861, 1], [57, [57]]]]), 1]
[63, List([[1200, 53467, [127, 1; 421, 1], [63, [21, 3]]], [1470, 80197, [13, 1; 31, 1; 199, 1], [63, [21, 3]]]]), 2]
[81, List([[1065, 42127, [103, 1; 409, 1], [81, [9, 9]]], [1308, 63511, [7, 1; 43, 1; 211, 1], [81, [9, 3, 3]]], 
 [1362, 68857, [37, 1; 1861, 1], [81, [9, 9]]]]), 3]
[84, List([[741, 20419, [7, 1; 2917, 1], [84, [42, 2]]]]), 1]
[93, List([[1038, 40021, [31, 1; 1291, 1], [93, [93]]]]), 1]
[108, List([[1389, 71611, [19, 1; 3769, 1], [108, [18, 6]]], [1443, 77281, [109, 1; 709, 1], [108, [18, 6]]]]), 2]
[129, List([[687, 17557, [97, 1; 181, 1], [129, [129]]]]), 1]
[144, List([[768, 21931, [7, 1; 13, 1; 241, 1], [144, [12, 12]]], [1092, 44287, [67, 1; 661, 1], [144, [12, 12]]], 
 [1713, 108871, [7, 1; 103, 1; 151, 1], [144, [12, 12]]], [1875, 130417, [7, 1; 31, 1; 601, 1], [144, [12, 12]]], 
 [2010, 149857, [277, 1; 541, 1], [144, [12, 12]]]]), 5]
[147, List([[1605, 95587, [61, 1; 1567, 1], [147, [21, 7]]], [1929, 138031, [97, 1; 1423, 1], [147, [147]]]]), 2]
[156, List([[1632, 98827, [37, 1; 2671, 1], [156, [78, 2]]]]), 1]
[171, List([[930, 32137, [7, 1; 4591, 1], [171, [57, 3]]], [2253, 188251, [7, 1; 26893, 1], [171, [57, 3]]]]), 2]
[183, List([[2199, 179341, [19, 1; 9439, 1], [183, [183]]]]), 1]
[192, List([[1146, 48769, [7, 1; 6967, 1], [192, [24, 8]]], [2118, 166381, [379, 1; 439, 1], [192, [24, 8]]], 
 [2415, 216277, [19, 1; 11383, 1], [192, [24, 8]]]]), 3]
[201, List([[1173, 51091, [19, 1; 2689, 1], [201, [201]]]]), 1]
[228, List([[849, 26791, [73, 1; 367, 1], [228, [114, 2]]], [1011, 37969, [43, 1; 883, 1], [228, [114, 2]]], 
 [1254, 58381, [79, 1; 739, 1], [228, [114, 2]]]]), 3]
[237, List([[1794, 119401, [139, 1; 859, 1], [237, [237]]], [2037, 153907, [13, 1; 11839, 1], [237, [237]]]]), 2]
[243, List([[1956, 141919, [139, 1; 1021, 1], [243, [27, 9]]]]), 1]
[252, List([[2658, 261961, [7, 1; 37423, 1], [252, [42, 6]]], 
 [3090, 353977, [13, 1; 73, 1; 373, 1], [252, [42, 6]]]]), 2]
[279, List([[2172, 174967, [13, 1; 43, 1; 313, 1], [279, [93, 3]]], 
 [2523, 236041, [13, 1; 67, 1; 271, 1], [279, [93, 3]]]]), 2]
[300, List([[2280, 192787, [7, 1; 27541, 1], [300, [30, 10]]]]), 1]
[324, List([[1335, 66157, [7, 1; 13, 1; 727, 1], [324, [18, 6, 3]]], 
 [1686, 105469, [7, 1; 13, 1; 19, 1; 61, 1], [324, [6, 6, 3, 3]]], 
 [3630, 488437, [37, 1; 43, 1; 307, 1], [324, [18, 6, 3]]], [3657, 495727, [337, 1; 1471, 1], [324, [18, 18]]]]), 4]
[351, List([[3387, 425257, [7, 1; 79, 1; 769, 1], [351, [39, 3, 3]]]]), 1]
[372, List([[3252, 392047, [61, 1; 6427, 1], [372, [186, 2]]], [3549, 466891, [31, 1; 15061, 1], [372, [186, 2]]], 
 [3738, 517921, [19, 1; 27259, 1], [372, [186, 2]]]]), 3]
[387, List([[3468, 445831, [397, 1; 1123, 1], [387, [129, 3]]]]), 1]
[444, List([[3144, 366451, [31, 1; 11821, 1], [444, [222, 2]]]]), 1]
[468, List([[2739, 278161, [13, 1; 21397, 1], [468, [78, 6]]], 
 [3792, 532987, [7, 1; 13, 1; 5857, 1], [468, [78, 6]]]]), 2]
[513, List([[2820, 294847, [7, 1; 73, 1; 577, 1], [513, [57, 3, 3]]], 
 [3225, 385567, [7, 1; 13, 1; 19, 1; 223, 1], [513, [57, 3, 3]]], 
 [3954, 579481, [7, 1; 19, 1; 4357, 1], [513, [57, 3, 3]]], 
 [4143, 636181, [7, 1; 13, 1; 6991, 1], [513, [57, 3, 3]]], [4467, 739537, [19, 1; 38923, 1], [513, [171, 3]]]]), 5]
[516, List([[2091, 162169, [7, 1; 23167, 1], [516, [258, 2]]], [4224, 661291, [163, 1; 4057, 1], [516, [258, 2]]]]), 2]
[549, List([[3333, 411811, [43, 1; 61, 1; 157, 1], [549, [183, 3]]]]), 1]
[576, List([[3009, 335671, [7, 1; 79, 1; 607, 1], [576, [24, 24]]], 
 [4764, 841111, [19, 1; 44269, 1], [576, [24, 24]]]]), 2]
[624, List([[3900, 563767, [199, 1; 2833, 1], [624, [78, 2, 2, 2]]]]), 1]
[651, List([[2388, 211471, [13, 1; 16267, 1], [651, [651]]], [2928, 317851, [19, 1; 16729, 1], [651, [651]]]]), 2]
[657, List([[1902, 134197, [7, 1; 19, 1; 1009, 1], [657, [219, 3]]], 
 [4359, 704221, [7, 1; 37, 1; 2719, 1], [657, [219, 3]]]]), 2]
[684, List([[4278, 678301, [13, 1; 52177, 1], [684, [114, 6]]]]), 1]
[687, List([[3819, 540601, [349, 1; 1549, 1], [687, [687]]]]), 1]
[732, List([[2307, 197377, [31, 1; 6367, 1], [732, [366, 2]]]]), 1]
[741, List([[2064, 158011, [7, 1; 22573, 1], [741, [741]]], [5844, 1265551, [7, 1; 180793, 1], [741, [741]]]]), 2]
[756, List([[4548, 766591, [7, 1; 97, 1; 1129, 1], [756, [42, 6, 3]]]]), 1]
[768, List([[3414, 432061, [7, 1; 61723, 1], [768, [48, 16]]]]), 1]
[777, List([[2901, 312019, [67, 1; 4657, 1], [777, [777]]]]), 1]
[804, List([[5520, 1129147, [79, 1; 14293, 1], [804, [402, 2]]]]), 1]
[819, List([[2874, 306241, [13, 1; 23557, 1], [819, [273, 3]]], [4845, 869947, [13, 1; 66919, 1], [819, [273, 3]]]]), 2]
[831, List([[4305, 686887, [193, 1; 3559, 1], [831, [831]]]]), 1]
[837, List([[4710, 822157, [7, 1; 67, 1; 1753, 1], [837, [93, 3, 3]]]]), 1]
[876, List([[4197, 652867, [181, 1; 3607, 1], [876, [438, 2]]]]), 1]
[900, List([[6465, 1548727, [229, 1; 6763, 1], [900, [30, 30]]]]), 1]
[912, List([[5115, 969577, [7, 1; 138511, 1], [912, [228, 4]]]]), 1]
[948, List([[3198, 379141, [7, 1; 54163, 1], [948, [474, 2]]]]), 1]
[972, List([[3576, 474019, [7, 1; 13, 1; 5209, 1], [972, [18, 18, 3]]]]), 1]
[999, List([[2469, 226051, [7, 1; 43, 1; 751, 1], [999, [111, 3, 3]]], 
 [2712, 272707, [19, 1; 31, 1; 463, 1], [999, [111, 3, 3]]]]), 2]
gp > length(L)
97
\end{verbatim}
}

%time = 43min, 29,719 ms.(Louboutin)
%time = 6h, 17min, 34,093 ms.(new)
%time = 4h, 47min, 6,172 ms.(old)

%%%%%%%%%%%%%%%%%%%%%%%%%%%%%%%%%%%%%%%%%%%%%%%%%%%%%%%%%%%%%%%%%%%%%%%%%%
\section{PARI/GP computations: The simplest cubic fields $L_m$ with class number $h_m=|\Cl(L_m)|<16$ for $-1\leq m\leq 10^7$}\label{S4}
{\footnotesize
\begin{verbatim}
gp > {
for(m=-1,10^7,
 K=bnfinit(x^3-m*x^2-(m+3)*x-1);
 cn=K.no;
 if(cn<16,print([[m,Mod(m,27)],m^2+3*m+9,factor(m^2+3*m+9),bnfcertify(K),K.clgp[1..2]]))
)
}
[[-1, Mod(26, 27)], 7, Mat([7, 1]), 1, [1, []]]
[[0, Mod(0, 27)], 9, Mat([3, 2]), 1, [1, []]]
[[1, Mod(1, 27)], 13, Mat([13, 1]), 1, [1, []]]
[[2, Mod(2, 27)], 19, Mat([19, 1]), 1, [1, []]]
[[3, Mod(3, 27)], 27, Mat([3, 3]), 1, [1, []]]
[[4, Mod(4, 27)], 37, Mat([37, 1]), 1, [1, []]]
[[5, Mod(5, 27)], 49, Mat([7, 2]), 1, [1, []]]
[[6, Mod(6, 27)], 63, [3, 2; 7, 1], 1, [3, [3]]]
[[7, Mod(7, 27)], 79, Mat([79, 1]), 1, [1, []]]
[[8, Mod(8, 27)], 97, Mat([97, 1]), 1, [1, []]]
[[9, Mod(9, 27)], 117, [3, 2; 13, 1], 1, [3, [3]]]
[[10, Mod(10, 27)], 139, Mat([139, 1]), 1, [1, []]]
[[11, Mod(11, 27)], 163, Mat([163, 1]), 1, [4, [2, 2]]]
[[12, Mod(12, 27)], 189, [3, 3; 7, 1], 1, [1, []]]
[[13, Mod(13, 27)], 217, [7, 1; 31, 1], 1, [3, [3]]]
[[14, Mod(14, 27)], 247, [13, 1; 19, 1], 1, [3, [3]]]
[[15, Mod(15, 27)], 279, [3, 2; 31, 1], 1, [3, [3]]]
[[16, Mod(16, 27)], 313, Mat([313, 1]), 1, [7, [7]]]
[[17, Mod(17, 27)], 349, Mat([349, 1]), 1, [4, [2, 2]]]
[[18, Mod(18, 27)], 387, [3, 2; 43, 1], 1, [3, [3]]]
[[19, Mod(19, 27)], 427, [7, 1; 61, 1], 1, [3, [3]]]
[[20, Mod(20, 27)], 469, [7, 1; 67, 1], 1, [3, [3]]]
[[21, Mod(21, 27)], 513, [3, 3; 19, 1], 1, [3, [3]]]
[[22, Mod(22, 27)], 559, [13, 1; 43, 1], 1, [3, [3]]]
[[23, Mod(23, 27)], 607, Mat([607, 1]), 1, [4, [2, 2]]]
[[24, Mod(24, 27)], 657, [3, 2; 73, 1], 1, [9, [3, 3]]]
[[25, Mod(25, 27)], 709, Mat([709, 1]), 1, [4, [2, 2]]]
[[26, Mod(26, 27)], 763, [7, 1; 109, 1], 1, [12, [6, 2]]]
[[27, Mod(0, 27)], 819, [3, 2; 7, 1; 13, 1], 1, [9, [3, 3]]]
[[28, Mod(1, 27)], 877, Mat([877, 1]), 1, [7, [7]]]
[[29, Mod(2, 27)], 937, Mat([937, 1]), 1, [4, [2, 2]]]
[[30, Mod(3, 27)], 999, [3, 3; 37, 1], 1, [3, [3]]]
[[31, Mod(4, 27)], 1063, Mat([1063, 1]), 1, [13, [13]]]
[[32, Mod(5, 27)], 1129, Mat([1129, 1]), 1, [7, [7]]]
[[33, Mod(6, 27)], 1197, [3, 2; 7, 1; 19, 1], 1, [9, [3, 3]]]
[[34, Mod(7, 27)], 1267, [7, 1; 181, 1], 1, [9, [3, 3]]]
[[35, Mod(8, 27)], 1339, [13, 1; 103, 1], 1, [9, [3, 3]]]
[[36, Mod(9, 27)], 1413, [3, 2; 157, 1], 1, [12, [6, 2]]]
[[38, Mod(11, 27)], 1567, Mat([1567, 1]), 1, [7, [7]]]
[[39, Mod(12, 27)], 1647, [3, 3; 61, 1], 1, [1, []]]
[[40, Mod(13, 27)], 1729, [7, 1; 13, 1; 19, 1], 1, [9, [3, 3]]]
[[41, Mod(14, 27)], 1813, [7, 2; 37, 1], 1, [3, [3]]]
[[42, Mod(15, 27)], 1899, [3, 2; 211, 1], 1, [12, [6, 2]]]
[[43, Mod(16, 27)], 1987, Mat([1987, 1]), 1, [7, [7]]]
[[44, Mod(17, 27)], 2077, [31, 1; 67, 1], 1, [12, [6, 2]]]
[[45, Mod(18, 27)], 2169, [3, 2; 241, 1], 1, [12, [6, 2]]]
[[47, Mod(20, 27)], 2359, [7, 1; 337, 1], 1, [9, [3, 3]]]
[[48, Mod(21, 27)], 2457, [3, 3; 7, 1; 13, 1], 1, [9, [3, 3]]]
[[49, Mod(22, 27)], 2557, Mat([2557, 1]), 1, [7, [7]]]
[[52, Mod(25, 27)], 2869, [19, 1; 151, 1], 1, [9, [3, 3]]]
[[53, Mod(26, 27)], 2977, [13, 1; 229, 1], 1, [9, [3, 3]]]
[[54, Mod(0, 27)], 3087, [3, 2; 7, 3], 1, [1, []]]
[[55, Mod(1, 27)], 3199, [7, 1; 457, 1], 1, [12, [6, 2]]]
[[57, Mod(3, 27)], 3429, [3, 3; 127, 1], 1, [12, [6, 2]]]
[[59, Mod(5, 27)], 3667, [19, 1; 193, 1], 1, [12, [6, 2]]]
[[66, Mod(12, 27)], 4563, [3, 3; 13, 2], 1, [1, []]]
[[67, Mod(13, 27)], 4699, [37, 1; 127, 1], 1, [12, [6, 2]]]
[[75, Mod(21, 27)], 5859, [3, 3; 7, 1; 31, 1], 1, [9, [3, 3]]]
[[84, Mod(3, 27)], 7317, [3, 3; 271, 1], 1, [9, [3, 3]]]
[[90, Mod(9, 27)], 8379, [3, 2; 7, 2; 19, 1], 1, [9, [3, 3]]]
[[93, Mod(12, 27)], 8937, [3, 3; 331, 1], 1, [1, []]]
[[100, Mod(19, 27)], 10309, [13, 2; 61, 1], 1, [3, [3]]]
[[103, Mod(22, 27)], 10927, [7, 2; 223, 1], 1, [9, [3, 3]]]
[[120, Mod(12, 27)], 14769, [3, 3; 547, 1], 1, [4, [2, 2]]]
[[139, Mod(4, 27)], 19747, [7, 2; 13, 1; 31, 1], 1, [9, [3, 3]]]
[[147, Mod(12, 27)], 22059, [3, 3; 19, 1; 43, 1], 1, [3, [3]]]
[[152, Mod(17, 27)], 23569, [7, 2; 13, 1; 37, 1], 1, [9, [3, 3]]]
[[154, Mod(19, 27)], 24187, [19, 2; 67, 1], 1, [3, [3]]]
[[174, Mod(12, 27)], 30807, [3, 3; 7, 1; 163, 1], 1, [3, [3]]]
[[188, Mod(26, 27)], 35917, [7, 2; 733, 1], 1, [12, [6, 2]]]
[[201, Mod(12, 27)], 41013, [3, 3; 7, 2; 31, 1], 1, [3, [3]]]
[[204, Mod(15, 27)], 42237, [3, 2; 13, 1; 19, 2], 1, [9, [3, 3]]]
[[228, Mod(12, 27)], 52677, [3, 3; 1951, 1], 1, [4, [2, 2]]]
[[235, Mod(19, 27)], 55939, [13, 2; 331, 1], 1, [12, [6, 2]]]
[[237, Mod(21, 27)], 56889, [3, 3; 7, 2; 43, 1], 1, [9, [3, 3]]]
[[255, Mod(12, 27)], 65799, [3, 3; 2437, 1], 1, [7, [7]]]
[[269, Mod(26, 27)], 73177, [13, 2; 433, 1], 1, [12, [6, 2]]]
[[271, Mod(1, 27)], 74263, [7, 1; 103, 2], 1, [3, [3]]]
[[286, Mod(16, 27)], 82663, [7, 3; 241, 1], 1, [1, []]]
[[309, Mod(12, 27)], 96417, [3, 3; 3571, 1], 1, [7, [7]]]
[[374, Mod(23, 27)], 141007, [37, 2; 103, 1], 1, [9, [3, 3]]]
[[390, Mod(12, 27)], 153279, [3, 3; 7, 1; 811, 1], 1, [9, [3, 3]]]
[[397, Mod(19, 27)], 158809, [7, 3; 463, 1], 1, [1, []]]
[[398, Mod(20, 27)], 159607, [7, 1; 151, 2], 1, [3, [3]]]
[[417, Mod(12, 27)], 175149, [3, 3; 13, 1; 499, 1], 1, [9, [3, 3]]]
[[498, Mod(12, 27)], 249507, [3, 3; 9241, 1], 1, [13, [13]]]
[[577, Mod(10, 27)], 334669, [43, 2; 181, 1], 1, [12, [6, 2]]]
[[629, Mod(8, 27)], 397537, [7, 3; 19, 1; 61, 1], 1, [3, [3]]]
[[716, Mod(14, 27)], 514813, [13, 1; 199, 2], 1, [12, [6, 2]]]
[[740, Mod(11, 27)], 549829, [7, 4; 229, 1], 1, [3, [3]]]
[[844, Mod(7, 27)], 714877, [37, 1; 139, 2], 1, [12, [6, 2]]]
[[876, Mod(12, 27)], 770013, [3, 3; 19, 2; 79, 1], 1, [3, [3]]]
[[911, Mod(20, 27)], 832663, [13, 3; 379, 1], 1, [1, []]]
[[939, Mod(21, 27)], 884547, [3, 3; 181, 2], 1, [3, [3]]]
[[972, Mod(0, 27)], 947709, [3, 2; 7, 3; 307, 1], 1, [9, [3, 3]]]
[[1083, Mod(3, 27)], 1176147, [3, 3; 7, 3; 127, 1], 1, [3, [3]]]
[[1119, Mod(12, 27)], 1255527, [3, 3; 7, 2; 13, 1; 73, 1], 1, [9, [3, 3]]]
[[1259, Mod(17, 27)], 1588867, [7, 1; 61, 3], 1, [1, []]]
[[1283, Mod(14, 27)], 1649947, [13, 3; 751, 1], 1, [1, []]]
[[1315, Mod(19, 27)], 1733179, [7, 3; 31, 1; 163, 1], 1, [9, [3, 3]]]
[[1426, Mod(22, 27)], 2037763, [7, 3; 13, 1; 457, 1], 1, [12, [6, 2]]]
[[1497, Mod(12, 27)], 2245509, [3, 3; 7, 1; 109, 2], 1, [3, [3]]]
[[1598, Mod(5, 27)], 2558407, [19, 3; 373, 1], 1, [1, []]]
[[1658, Mod(11, 27)], 2753947, [7, 4; 31, 1; 37, 1], 1, [9, [3, 3]]]
[[1769, Mod(14, 27)], 3134677, [7, 3; 13, 1; 19, 1; 37, 1], 1, [9, [3, 3]]]
[[2344, Mod(22, 27)], 5501377, [7, 3; 43, 1; 373, 1], 1, [12, [6, 2]]]
[[2361, Mod(12, 27)], 5581413, [3, 3; 37, 2; 151, 1], 1, [12, [6, 2]]]
[[2389, Mod(13, 27)], 5714497, [19, 1; 67, 3], 1, [1, []]]
[[3108, Mod(3, 27)], 9668997, [3, 3; 13, 3; 163, 1], 1, [3, [3]]]
[[3480, Mod(24, 27)], 12120849, [3, 2; 13, 3; 613, 1], 1, [9, [3, 3]]]
[[4059, Mod(9, 27)], 16487667, [3, 2; 7, 5; 109, 1], 1, [9, [3, 3]]]
[[4170, Mod(12, 27)], 17401419, [3, 3; 7, 3; 1879, 1], 1, [4, [2, 2]]]
[[5088, Mod(12, 27)], 25903017, [3, 3; 7, 3; 2797, 1], 1, [4, [2, 2]]]
[[5258, Mod(20, 27)], 27662347, [19, 3; 37, 1; 109, 1], 1, [3, [3]]]
[[5305, Mod(13, 27)], 28158949, [7, 1; 13, 3; 1831, 1], 1, [12, [6, 2]]]
[[5677, Mod(7, 27)], 32245369, [13, 4; 1129, 1], 1, [12, [6, 2]]]
[[6816, Mod(12, 27)], 46478313, [3, 3; 7, 2; 19, 1; 43, 2], 1, [9, [3, 3]]]
[[7502, Mod(23, 27)], 56302519, [7, 2; 13, 3; 523, 1], 1, [3, [3]]]
[[7837, Mod(7, 27)], 61442089, [37, 3; 1213, 1], 1, [1, []]]
[[8457, Mod(6, 27)], 71546229, [3, 2; 19, 4; 61, 1], 1, [9, [3, 3]]]
[[10927, Mod(19, 27)], 119432119, [19, 1; 31, 3; 211, 1], 1, [3, [3]]]
[[12117, Mod(21, 27)], 146858049, [3, 3; 13, 1; 19, 3; 61, 1], 1, [9, [3, 3]]]
[[12745, Mod(1, 27)], 162473269, [7, 6; 1381, 1], 1, [1, []]]
[[14349, Mod(12, 27)], 205936857, [3, 3; 7, 3; 37, 1; 601, 1], 1, [12, [6, 2]]]
[[24614, Mod(17, 27)], 605922847, [43, 3; 7621, 1], 1, [7, [7]]]
[[36435, Mod(12, 27)], 1327618539, [3, 3; 13, 3; 22381, 1], 1, [13, [13]]]
[[70509, Mod(12, 27)], 4971730617, [3, 3; 7, 1; 31, 3; 883, 1], 1, [9, [3, 3]]]
[[91858, Mod(4, 27)], 8438167747, [73, 3; 109, 1; 199, 1], 1, [9, [3, 3]]]
[[100952, Mod(26, 27)], 10191609169, [7, 1; 79, 3; 2953, 1], 1, [9, [3, 3]]]
[[101471, Mod(5, 27)], 10296668263, [7, 3; 5479, 2], 1, [4, [2, 2]]]
[[222550, Mod(16, 27)], 49529170159, [7, 6; 73, 2; 79, 1], 1, [3, [3]]]
[[263135, Mod(20, 27)], 69240817639, [7, 3; 43, 3; 2539, 1], 1, [1, []]]
[[266748, Mod(15, 27)], 71155295757, [3, 2; 7, 2; 13, 3; 271, 2], 1, [9, [3, 3]]]
[[506370, Mod(12, 27)], 256412096019, [3, 3; 193, 3; 1321, 1], 1, [1, []]]
[[612049, Mod(13, 27)], 374605814557, [13, 3; 19, 3; 24859, 1], 1, [7, [7]]]
[[1360624, Mod(13, 27)], 1851301751257, [7, 4; 43, 1; 61, 3; 79, 1], 1, [9, [3, 3]]]
[[1376233, Mod(16, 27)], 1894021398997, [7, 1; 13, 3; 43, 3; 1549, 1], 1, [3, [3]]]
[[6440111, Mod(17, 27)], 41475049012663, [7, 3; 229, 3; 10069, 1], 1, [4, [2, 2]]]
\end{verbatim}
}

\end{document}